\documentclass[10pt,bezier,amstex]{article}  % include bezier curves
\usepackage{color}              % Need the color package
\usepackage{xcolor}              % Need the color package
\usepackage{graphicx}
\usepackage[]{amsmath}
\usepackage{amssymb}
\usepackage{hyperref,bbm}
\usepackage{float}
\definecolor{MyDarkBlue}{rgb}{0,0.08,0.50}
\definecolor{BrickRed}{rgb}{0.65,0.08,0}
\usepackage{amsmath}

\hypersetup{
colorlinks=true,       % false: boxed links; true: colored links
    linkcolor=MyDarkBlue,          % color of internal links
    citecolor=BrickRed,        % color of links to bibliography
    filecolor=red,      % color of file links
    urlcolor=cyan           % color of external links
}

\usepackage{enumerate}
\usepackage[mathscr]{eucal}
\usepackage{graphicx}
\usepackage{latexsym}
\usepackage{amssymb}
\usepackage{psfrag}
\usepackage{epsfig}
\usepackage{tikz}
\usetikzlibrary{automata,positioning}

\usepackage{amsfonts}

\usepackage{graphicx}

\usepackage{amsfonts}

\usepackage{color}

\newtheorem{Lemma}{Lemma}[section]
\newtheorem{Theorem}[Lemma]{Theorem}
\newtheorem{Assumption}{Assumption}

\newtheorem{Proposition}[Lemma]{Proposition}

\newtheorem{Corollary}[Lemma]{Corollary}

\newtheorem{remark}[Lemma]{Remark}

\newcommand{\beq}{\begin{eqnarray*}}
\newcommand{\eeq}{\end{eqnarray*}}
\newcommand{\beqn}{\begin{eqnarray}}
\newcommand{\eeqn}{\end{eqnarray}}

\newcommand{\bt}{\begin{Theorem}}
\newcommand{\et}{\end{Theorem}}

\newcommand{\bas}{\begin{Assumption}}
\newcommand{\eas}{\end{Assumption}}

\newcommand{\be}{\begin{equation}}
\newcommand{\ee}{\end{equation}}

\newcommand{\mb}[1]{{\mathbf #1}}

\numberwithin{equation}{section}

\definecolor{darkgreen}{rgb}{0,.4,0}
\definecolor{darkagenta}{rgb}{.5,0,.5}
\definecolor{darkred}{rgb}{1,0,0}%was 0.85
\definecolor{darkblue}{rgb}{0,0,.4}

\newcommand{\affiliation}[1]{%
    \par {\raggedright #1 \vspace{1.2em}\noindent\par}
}

\usepackage{array}
\begin{document}

\tikzset{every node/.style={auto}}
 \tikzset{every state/.style={rectangle, minimum size=0pt, draw=none, font=\normalsize}}
 \tikzset{bend angle=20}

	\author{
    Ashwin K. Seshadri\textsuperscript{a,*} and
    Abhishek Pal Majumder\textsuperscript{b}
}

\title{Weighted spatial averages with missing observations: Convergence and accuracy of delta-method variance approximations}
\date{}

\maketitle

\affiliation{%
\textsuperscript{a}Centre for Atmospheric and Oceanic Sciences and Divecha Centre for Climate Change, Indian Institute of Science, Bengaluru 560012, India.\\
\textsuperscript{*}Corresponding author. Email: \href{mailto:ashwins@iisc.ac.in}{ashwins@iisc.ac.in}%
}

\affiliation{%
\textsuperscript{b}School of Mathematics, Statistics and Physics, Newcastle University, Newcastle upon Tyne NE1 7RU, United Kingdom.\\
Email: \href{mailto:Abhishek.Pal-Majumder@newcastle.ac.uk}{Abhishek.Pal-Majumder@newcastle.ac.uk}%
}

\begin{abstract}
Weighted spatial averages from incomplete networks are ratios whose variance is approximated by truncating a Taylor series (the delta method). Three questions remain open: whether the series converges, how large a truncation's error is, and how it behaves as the network grows. We address these for observations missing completely at random, independently with a common probability $\alpha$. First, the truncations, which count corrections for the denominator (the total reported weight) and for its coupling with the numerator, converge for a given network only if $\alpha>\tfrac12$ and both corrections increase together. Second, whether or not the series converges, the error of truncations correcting only for the denominator is controlled at any network size by its coefficient of variation $\rho$, with $\rho^2=\tfrac{1-\alpha}{\alpha}\sum_i\beta_i^2$ for weights $\beta_i$, and by the probability of rare reporting patterns. Third, for uniform weights and bounded fields, these errors vanish as the network grows, at any fixed availability. If correlations also decay with distance, the first-order approximation of earlier work has error of order $\rho^4$, against $\rho^2$ for the simplest approximation. Only truncations of equal order in numerator and denominator, beginning with the first-order truncation, avoid spurious dependence on the mean level; higher ones need not be more accurate. The exact variance for uniform weights, a single sum, shows that for correlated fields the first-order approximation overstates the contribution of missing data but is conservative overall. Concentrated weights degrade it more than $\rho$ indicates. Rainfall simulations over India illustrates results, several of which extend to self-normalized estimators.
\end{abstract}

\section*{Keywords}
Ratio estimator; Delta method; Missing completely at random; Self-normalized estimator; Sampling error; Spatial dependence; Rain gauge networks

\section{Introduction}

Spatial averages of geophysical variables, including temperature and rainfall, are often estimated from observations that are incomplete or irregular in space \cite{Rodriguez-Iturbe1974,Shen1994,Morrissey1995,Villarini2008,Prakash2019}. The uncertainty from sampling of potential observation sites, including effects of missing observations, is important for climate and hydrological applications \cite{Kagan1997}. This paper examines how random unavailability of observations affects the variance of the estimated spatial mean, mainly when observations are missing completely at random (MCAR), so that reporting is independent of the underlying field.

This uncertainty is often estimated by simulation. Global temperature products, for example, include a contribution to uncertainty from incomplete spatial coverage \cite{Morice2021}, and the effect of missing observations can be estimated by withholding observations from a more complete reference field \cite{CowtanWay2014,Cowtan2018} or by Monte Carlo simulations of alternate plausible fields \cite{Lenssen2024}. Simulations have also been used to propagate and infer uncertainty in hydrological applications \cite{Hong2006,Vrugt2008} and to assess the sampling error of area-averaged rainfall by subsampling dense rain-gauge networks \cite{Villarini2008,Prakash2019}. Although simulation-based approaches can be versatile, the calculations must be repeated across configurations and it is challenging to explicitly resolve how uncertainty depends on the availability, weighting, or spatial structure. Analytical expressions exist for the standard error of simple spatial averages from fixed gauge networks \cite{Morrissey1995}, and low-order delta-method approximations have been derived for weighted averages with missing data and compared with simulation \cite{Seshadri2018}, and used with site-specific reporting probabilities \cite{Chen2025}. However, the convergence of these approximations, their higher-order structure, and their truncation error have not been characterized previously in the literature.

The broader difficulty is a familiar one, and not confined to spatial averages. In survey sampling, ratio and H\'ajek estimators are routinely assigned a linearized variance, which is recognized as being adequate only when the coefficient of variation of the denominator is small \cite{Cochran1977,Sarndal1992}. Moreover, for the ratio estimator under simple random sampling, the Taylor expansion need not even converge for every sample, but the usual first-order approximations to its bias and mean squared error can nevertheless be justified with bounds on their error \cite{DavidSukhatme1974}, while higher-order approximations have also been developed \cite{Tin1965}. Self-normalized importance sampling and ratio metrics in online experiments rely on the same first-order approximation \cite{Hesterberg1995,Deng2018}, even though for ratios of normal variables the moments may not exist, and delta-method intervals can fail when the denominator is imprecise \cite{Geary1930,Fieller1954,Hinkley1969,Hirschberg2010}. Similarly, Taylor-series approximations to expected utility can diverge, yet low-order truncations are often reasonably accurate \cite{Loistl1976,Hlawitschka1994,Markowitz2014}, reflecting the general behavior of asymptotic expansions, whose accuracy at fixed sample size need not improve as terms are added \cite{BarndorffNielsenCox1989,Hall1992}. The present paper makes these distinctions precise for weighted spatial averages under MCAR, in which the number of reporting sites is random and may be zero, the weights may be unequal, and the field values are spatially correlated. Although the motivation is geophysical, several of the results bear on self-normalized estimators in general, since these results use only the fact that the estimator is a convex combination of the observed values with random weights.

We adopt the setting of \cite{Seshadri2018}. Let $r_i(t)$ denote the field value at site $i$, $\beta_i>0$ its prescribed weight, with $\sum_i\beta_i=1$, and $s_i(t)\in\{0,1\}$ its reporting status, so that a site enters the spatial average only when its observation is present. With $\mb{s}:=\{s_i(t):i=1,\ldots,N\}$ the reporting status across observation sites, the weighted ratio estimator is defined as
\begin{equation}
\hat{r}\left(t\right):=\frac{R\left(t\right)}{S\left(t\right)}=\frac{\sum_{i=1}^{N}\beta_{i}s_{i}\left(t\right)r_{i}\left(t\right)}{\sum_{i=1}^{N}\beta_{i}s_{i}\left(t\right)}.
\label{eq:p2}
\end{equation}
When all sites report, $S=1$ and $\widehat r$ reduces to the ordinary weighted mean. With missing observations, both the weighted sum of reported values $R$ and the total reported weight $S$ are random variables, and different reporting patterns give different values of $\widehat r(t)$ even when the field statistics are fixed. The randomness in reporting may have physical or technical origins, such as missing, unreported, or irregularly timed observations \cite{Vinnikov2004,Seshadri2018}, or arise by design, as in deliberate subsampling of a dense network to evaluate the uncertainty of an incomplete one \cite{Prakash2019}. These applications share the same statistical structure and are treated together here.

We assume that sites report independently with a common probability $\alpha$ of reporting, the availability, so that the $s_i(t)$ are i.i.d.\ Bernoulli$(\alpha)$ with $0<\alpha\le1$, and that the reporting vector is independent of the field. Missingness that does not depend on the field is called missing completely at random, in the terminology that originates with \cite{Rubin1976b}. The model used here is the special case with independent reporting and a common $\alpha$, which we refer to simply as MCAR. Because all observations may be absent whenever $\alpha<1$, we define the all-missing event $A_{\emptyset}:=\{S=0\}$. On this event we assign the estimator a bounded value $r_*$; away from it, the estimator is the weighted mean of the reported observations:
\[
\widehat{r}(t) =
\begin{cases}
r_*(t), & \text{if } A_{\emptyset} \text{ occurs}, \\[4pt]
\dfrac{\sum_{i=1}^{N}\beta_i s_i(t) r_i(t)}{\sum_{i=1}^{N}\beta_i s_i(t)}, & \text{otherwise.}
\end{cases}
\]
Under MCAR, $\mathbf{P}(A_{\emptyset})=(1-\alpha)^N$, so at fixed $\alpha>0$ this convention plays a role only on an event whose probability decreases exponentially as $N$ grows. It is nevertheless retained explicitly because it matters for very small networks or very low availability.

Ratio estimators have a long history \cite{Hartley1954,Goodman1958,Hinkley1969,Kendall1994}, and related forms arise in geophysical averaging \cite{Gandin1993,Kagan1997,Seshadri2018}. The estimator in Eq.~\eqref{eq:p2} is a self-normalized, H\'ajek-type estimator \cite{Horvitz1952,Hajek1971} wherein the reported observations are averaged using their realized total weight $\sum_{i=1}^{N}\beta_i s_i(t)$. Its moments can be evaluated exactly by averaging over the reporting patterns, but for arbitrary weights there are as many as $2^N$ of these, which is impractical for large networks. This paper instead uses the classical delta-method expansion of a ratio \cite{Oehlert1992}, which gives the low-order approximations as members of a common sequence.

The expansion contains two kinds of correction: terms in powers of the fluctuation of the denominator $S$ about its mean, and terms that couple these fluctuations with those of $R$. We label an approximation $(n_1,n_2)$ by the number of corrections of each kind that it retains. The simplest members are $(0,0)$, which treats the denominator as fixed at its mean, and $(0,1)$, which adds the first correction for its randomness. We call $(0,1)$ the first-order approximation; it is the variance estimator of \cite{Seshadri2018}. Under MCAR, each approximation depends on the field only through its first and second moments. For general weights the explicit formulas become increasingly cumbersome as the order rises, and we derive them for the low orders, with the full algebra provided in the Supplementary Information. For uniform weights, the exact variance reduces to a single sum over the number of reporting sites, which holds for any spatial covariance and needs no simulation. Uniform weights therefore serve two purposes: in practice they make the approximations unnecessary, and they provide an exact benchmark that indicates how accurate the approximations are for general weights, where exact evaluation is impractical.

Three related questions are distinguished in this paper. The first is convergence: whether the approximations approach the exact variance as more terms of the series are retained, for a network of given size $N$. The second is the truncation error, the difference between the exact variance and a given approximation for a given network, which is well defined whether or not the series converges. The third, which we call fixed-order accuracy, concerns how the truncation error of a given approximation behaves as the network grows. Both the truncation error and fixed-order accuracy are controlled mainly by the concentration of the random denominator, summarized by
\[
\rho:=\frac{\sqrt{\mathrm{Var}(S)}}{\mathbf{E}S},
\]
the coefficient of variation of the total reported weight. The second and third questions are the ones of practical relevance.

The main contributions are as follows:
\begin{enumerate}
\item Convergence and accuracy are distinct. Under MCAR the truncated variances converge only if $\alpha>\tfrac12$ and the two kinds of correction are increased together, whereas the truncation error of every $(0,n_2)$ approximation is bounded at each network size through $\rho$ and, for uniform weights and a bounded field, vanishes as the network grows at any fixed availability (Theorem~\ref{thm:rho-error}, Corollary~\ref{cor:mar-error}, and Table~\ref{tab:three-results}).
\item For uniform weights and correlations that decay with distance, the error of the first-order approximation is of order $\rho^4$, compared with $\rho^2$ for the simplest approximation, which helps explain why the estimator of \cite{Seshadri2018} has performed well (Proposition~\ref{prop:rho4}).
\item Only truncations that treat numerator and denominator to the same order, which we call shift-invariant, avoid a spurious dependence on the mean level of the field, and even these need not improve on the first-order approximation (Section~\ref{sec:shift} and Table~\ref{tab:truncation-properties}).
\item For uniform weights, the exact variance separates into the variance of the average with complete reporting and a contribution of missing data. Relative to that contribution, the error of the first-order approximation is set by a dimensionless ratio $\lambda$ of the variance of the average with complete reporting to the expected spatial variance of a single realization, which is small under local dependence but not for strongly correlated fields. Relative to the total variance, the error stays of order $\rho^2$ whatever the spatial dependence, and for correlated fields the approximation is conservative (Proposition~\ref{prop:uniform-decomp} and Remark~\ref{rem:lambda}).
\item Stress tests show that concentrated weights, whose effect $\rho$ only partly reflects, can make the first-order approximation unreliable even at moderate $\rho$. Together, these results indicate when the delta method can be trusted for spatial averages and when an exact calculation is preferable, including a rule of thumb for the number of sites needed for a given accuracy (Table~\ref{tab:practical}).
\end{enumerate}

Three further points should be noted. First, $\mathrm{Var}(\widehat r)$ is the variance of $\widehat r(t)$ over time, arising jointly from the variability of the field and of the reporting process, whose statistics are taken to be constant in time. No assumption about temporal dependence is needed, and we omit the argument $t$ when it is evident. Second, if the field is instead frozen in time, so that only reporting varies, the same framework gives the variance due to sampling alone, and for uniform weights the two sources of variability can be separated exactly (Proposition~\ref{prop:uniform-decomp}). Third, the expansion, the role of $\rho$, and the error bound carry over to sites that report independently with different probabilities, including probabilities that depend on observed covariates, so that the analysis is not limited to the MCAR model here (Remark~\ref{prop:het-alpha}).

Section~\ref{sec:moments} develops the theory for the three questions, and bounds the truncation error through the concentration of the denominator. Section~\ref{sec:truncated} specializes to MCAR, giving explicit low-order formulas, identifying the truncations that are invariant under a common shift of the field, and deriving the exact variance for uniform weights. Section~\ref{sec:numerics} tests these results against exact calculations, including for the IMD rainfall field and under various stress conditions, and Section~\ref{sec:discussion} discusses their implications and limitations. The proofs are given in Appendices~\ref{app:tools}--\ref{app:proof-rho4}, and the Supplementary Information gives the derivations of the truncated MCAR moments and additional numerical checks.

\section{Convergence, truncation error, and fixed-order accuracy}
\label{sec:moments}

Because the reporting vector $\mathbf{s}$ is random, the estimator $\widehat r=R/S$ is a ratio with a random denominator. This section expands the ratio in a Taylor series about the means of its numerator and denominator, and addresses the three questions set out in the Introduction: when the series converges, how large the error of a finite truncation is, and how that error behaves as the network grows. The analysis applies to a finite network with any spatial covariance.

The expansion is organized by the normalized denominator fluctuation
\[
\frac{S-\mathbf{E}S}{\mathbf{E}S}.
\]
Its typical size is measured by its root-mean-square value, the coefficient of variation of $S$,
\begin{equation}
\rho:=\frac{\|S-\mathbf{E}S\|_2}{\mathbf{E}S}
=\frac{\sqrt{\mathrm{Var}(S)}}{\mathbf{E}S},
\label{eq:def-rho-general}
\end{equation}
which is the main parameter controlling the accuracy of finite truncations (Section~\ref{sec:accuracy}).

Under MCAR, $s_i\stackrel{\mathrm{i.i.d.}}{\sim}\mathrm{Bernoulli}(\alpha)$ and $\sum_i\beta_i=1$, so
\[
\mathbf{E}S=\alpha,
\qquad
\mathrm{Var}(S)=\alpha(1-\alpha)W_2,
\qquad
W_2:=\sum_i\beta_i^2,
\]
and hence
\begin{equation}
\rho=\sqrt{\frac{1-\alpha}{\alpha}}\,\sqrt{W_2}.
\label{eq:rho-mar-intro}
\end{equation}
The quantity $1/W_2$ is Kish's effective sample size for unequal weights \cite{Kish1965}: it equals $N$ for uniform weights and decreases as weight is concentrated on fewer sites. Thus $\rho^2$ is the ratio of the odds $(1-\alpha)/\alpha$ that a site does not report to the effective number of sites, and the denominator is concentrated when availability is high and the weight is spread over many sites. For uniform weights, $\rho^2=(1-\alpha)/(\alpha N)$.

\begin{table}[t]
\centering
\small
\renewcommand{\arraystretch}{1.25}
\begin{tabular}{c p{0.6\linewidth}}
\hline
Symbol & Meaning \\
\hline
$N$ & number of sites \\
$r_i(t)$ & field value at site $i$ (e.g.\ rainfall, temperature) \\
$m_i,\ \sigma_{ij}$ & site means $m_i=\mathbf{E}r_i$ and covariances $\sigma_{ij}=\mathrm{Cov}(r_i,r_j)$, with $\sigma_i^2=\sigma_{ii}$ \\
$\beta_i$ & weight of site $i$ (e.g.\ area fraction); $\beta_i>0$, $\sum_i\beta_i=1$ \\
$s_i(t)$ & reporting indicator; independent $\mathrm{Bernoulli}(\alpha)$ under MCAR \\
$\alpha$ & availability: probability that a site reports \\
$K=\sum_i s_i$ & number of reporting sites \\
$R=\sum_i\beta_i s_i r_i$ & numerator: weighted sum of reported values \\
$S=\sum_i\beta_i s_i$ & denominator: total weight of reporting sites \\
$\widehat r=R/S$ & spatial average on $S>0$; a bounded value $r_*$ on the all-missing event $A_\emptyset=\{S=0\}$ \\
$m_R,\ m_S$ & $m_R=\mathbf{E}R$ and $m_S=\mathbf{E}S$; under MCAR, $m_S=\alpha$ and $m_R=\alpha m_r$ \\
$m_r=\sum_i\beta_im_i$ & weighted mean of the site means \\
$\bar r=\sum_i\beta_ir_i$ & average with complete reporting \\
$\mathrm{Var}(\widehat r)-\mathrm{Var}(\bar r)$ & contribution of missing data \\
$U=R-m_R,\ V=S-m_S$ & numerator and denominator fluctuations \\
$W_2=\sum_i\beta_i^2$ & weight concentration; $1/W_2$ is the effective number of sites ($=N$ for uniform weights) \\
$\rho=\sqrt{\mathrm{Var}(S)}/m_S$ & coefficient of variation of $S$; $\rho^2=\tfrac{1-\alpha}{\alpha}W_2$ under MCAR \\
$\mu_2^{(n_1,n_2)}$ & variance approximation retaining $n_1$ coupling and $n_2$ denominator corrections \\
$\mu_{2,1}=\mathbf{E}[U^2V]$ & mixed moment: dependence of numerator variability on the reported weight \\
$\mu_{1,2}=\mathbf{E}[UV^2]$ & mixed moment: displacement of the numerator by large denominator fluctuations \\
$\kappa=\mathrm{Var}(UV)$ & variance of the numerator--denominator product \\
$s^2,\ \lambda$ & spatial variance of a realization, $s^2=N^{-1}\sum_i(r_i-\bar r)^2$, and $\lambda=\mathrm{Var}(\bar r)/\mathbf{E}s^2$ (uniform weights) \\
\hline
\end{tabular}
\caption{Main notation used in the paper. The argument $t$ is omitted when evident.}
\label{tab:notation}
\end{table}

The first two subsections derive the expansion and establish when the series and its moments converge, and Section~\ref{sec:accuracy} bounds the truncation error, whether or not the series converges, and derives the fixed-order accuracy of the truncations. Much of this analysis does not depend on the missing-data model. The convergence conditions, the error bound of Theorem~\ref{thm:rho-error}, and the invariance under a common shift (Section~\ref{sec:shift}) require only that $\widehat r$ be a weighted average of the observed values with random nonnegative weights that sum to one. The MCAR assumption is required only for the explicit thresholds, formulas, and rates, and the Bernstein bound requires independent reporting. Table~\ref{tab:notation} summarizes the notation.

\subsection{Taylor series expansion of the ratio function}

We expand $f(R,S)=R/S$ about $(\mathbf{E}R,\mathbf{E}S)$, following the classical delta method, and later retain only a finite number of terms. With $x=(S-\mathbf{E}S)/\mathbf{E}S$, the ratio is a geometric series in $-x$,
\[
\frac{R}{S}=\frac{R}{\mathbf{E}S}\,\frac{1}{1+x}=\frac{R}{\mathbf{E}S}\sum_{n=0}^{\infty}(-x)^n,
\]
and writing $R=\mathbf{E}R+(R-\mathbf{E}R)$ splits it into two sums,
\begin{eqnarray}
f(R,S)
&=& \sum_{n=0}^{\infty} (-1)^{n} 
      \frac{R - \mathbf{E}R}{(\mathbf{E}S)^{n+1}}(S - \mathbf{E}S)^{n}
  + \sum_{n=0}^{\infty} (-1)^{n}
      \frac{\mathbf{E}R}{(\mathbf{E}S)^{n+1}}(S - \mathbf{E}S)^{n}.
\label{eq:p12}
\end{eqnarray}
This is the Taylor series of $f$ about $(\mathbf{E}R,\mathbf{E}S)$: since $f$ is linear in $R$, every derivative containing two or more $R$-derivatives vanishes, and only these two families of terms remain. The first sum carries the fluctuation $R-\mathbf{E}R$ of the numerator and the second its mean $\mathbf{E}R$, and in both sums the higher-order terms involve higher powers of $x$. These representations hold only where the series converges, a condition examined in Section~\ref{sec:conv}. The variance approximations are built from the moments of these series.

\subsection{Convergence of the variance series}
\label{sec:conv}

We first consider convergence, before examining the error of finite truncations. The finite truncations $\mu_2^{(n_1,n_2)}=\mathbf{E}[\mathcal T_{n_1,n_2}[f]^2]$, defined in Section~\ref{sec:accuracy}, are expectations of explicit polynomials in $(R,S)$ and remain well defined whether or not the full series converges. However, convergence of the series itself, and of its moments, involves two requirements. First, the series in Eq.~\eqref{eq:p12} must converge for the realized pair $(R,S)$. This is determined by the factor $|(S-\mathbf{E}S)/\mathbf{E}S|$, and absolute convergence holds when this is strictly less than one, as for any geometric series \cite{Rudin1976}. Second, the expectations of the partial sums must converge, and if the series converges absolutely and its partial sums are dominated by an integrable function, dominated convergence ensures this \cite{Weir1973}. Together, these conditions ensure that the moment series obtained from Eq.~\eqref{eq:p12} converge to the moments of $\widehat r$.

All convergence statements in this section therefore refer to reporting patterns in which at least one site reports, $S>0$. On the all-missing event $A_\emptyset=\{S=0\}$ the series cannot represent the estimator, since neither of the two sums in Eq.~\eqref{eq:p12} converges there (Remark~\ref{rem:trunc-conv}). Instead, $\widehat r$ is assigned the bounded value $r_*$ described in the Introduction, and the contribution of this event to the truncation error is bounded explicitly in Section~\ref{sec:accuracy}.

\begin{Proposition}[Convergence on nonempty reporting patterns]\label{prop:var-conv}
Suppose that $m_S=\mathbf{E}S>0$ and that $\mathbf{E}r_i^2<\infty$ for every site. If, whenever at least one site reports, the denominator stays within a fixed fraction $\bar q<1$ of its mean,
\begin{equation}
\frac{|S-\mathbf{E}S|}{\mathbf{E}S}\le\bar q
\quad\text{almost surely on }\{S>0\},
\label{eq:suff-a.s}
\end{equation}
then the series in Eq.~\eqref{eq:p12} converges absolutely on every nonempty reporting pattern, and the second moment of $\widehat r$ on these patterns may be computed term by term. Under MCAR, condition~\eqref{eq:suff-a.s} holds if and only if $\alpha>\tfrac12$.
\end{Proposition}

The proof is given in Appendix~\ref{app:proofs-sec2}.

Under MCAR, the relative deviation $|S-\mathbf{E}S|/\mathbf{E}S$ above the mean is largest when every site reports, $S=1$, where it equals $(1-\alpha)/\alpha$, and any deviation below the mean is at most $1-\beta_{\min}/\alpha$, attained when only the site with the smallest weight $\beta_{\min}$ reports. The second is always below one, so condition~\eqref{eq:suff-a.s} holds exactly when $\alpha>\tfrac12$, with $\bar q$ the larger of the two. The condition $\alpha>\tfrac12$ cannot be relaxed. When $\alpha\le\tfrac12$, the deviation $(1-\alpha)/\alpha$ on the pattern in which every site reports is at least one, so the terms of the series do not shrink there and it does not converge, apart from the trivial case $R=0$. At $\alpha=\tfrac12$ the factor equals one, so the terms neither grow nor shrink, and the partial sums oscillate without converging (Table~\ref{tab:convergence}). Thus $\alpha>\tfrac12$ is both necessary and sufficient for convergence on every nonempty reporting pattern, although it is not needed for accurate finite truncations (Section~\ref{sec:accuracy}). Small $\rho$ does not ensure convergence either: $\rho$ measures the typical size of the normalized fluctuation, whereas convergence requires its largest value to stay below one. In a large network with $\alpha<\tfrac12$, for example, $\rho$ is small but the series still diverges on the pattern in which every site reports.

The field need not be nonnegative: on $S>0$, $\widehat r$ is a weighted average of the reported values, so that $|\widehat r|\le\max_i|r_i|$ whatever their signs.

\subsection{Truncation error and fixed-order accuracy, controlled by $\rho$}
\label{sec:accuracy}

Having considered when the series and its moments converge, we now turn to the quantity of most practical interest, the truncation error, the difference between the exact variance and a truncated approximation for a given network. It is well defined whether or not the series converges, and includes the contribution of rare reporting patterns on which the series fails. We then turn to fixed-order accuracy, which concerns how this error behaves as $\rho\to0$, for example as the network grows. The bounds below are first stated in terms of normalized moments of the denominator fluctuation, and are then specialized to MCAR at fixed availability with weights spread over many sites, where $\rho$ becomes the natural small parameter. A reader interested mainly in applying the approximations may focus on the statements of Theorem~\ref{thm:rho-error}, Corollary~\ref{cor:mar-error}, and Proposition~\ref{prop:rho4}, and then turn to the low-order formulas of Section~\ref{sec:truncated} and the numerical results of Section~\ref{sec:numerics}.

We write $U:=R-m_R$ and $V:=S-m_S$, where $m_R:=\mathbf{E}R$ and $m_S:=\mathbf{E}S$, and for $p\ge2$ we define
\[
\rho_p:=\left(\mathbf{E}\left|V/m_S\right|^p\right)^{1/p},
\qquad\text{so that}\qquad
\rho_2=\rho.
\]
Bounds on the higher-order terms involve these higher moments of the denominator, which under MCAR at fixed $\alpha$ are comparable to $\rho$ (Lemma~\ref{lem:term-order} below).

\paragraph{Centered expansion and truncation.}
Substituting $R=m_R+U$ and $S=m_S+V$ in Eq.~\eqref{eq:p12} yields two families of terms, $(-1)^n U V^n/m_S^{n+1}$ and $(-1)^n m_R V^n/m_S^{n+1}$. Subtracting the mean $\mathbf{E}f$ replaces $UV^n$ by $UV^n-\mathbf{E}(UV^n)$ and $V^n$ by $V^n-\mathbf{E}V^n$, giving the sum of two centered series,
\begin{align}
f - \mathbf{E}f
&= \sum_{n=0}^{\infty} (-1)^n
   \frac{U V^n - \mathbf{E}(UV^n)}{m_S^{\,n+1}}
 \;+\;
 \sum_{n=0}^{\infty} (-1)^{n}
   \frac{m_R\big(V^n-\mathbf{E}V^n\big)}{m_S^{\,n+1}},
\label{eq:sec5-center}
\end{align}
as detailed in Supplementary Information Section~S1. This representation involving infinite sums only holds where the series converges and termwise expectations are justified (Proposition~\ref{prop:var-conv}). We call the first sum the mixed series and the second the denominator series.

The $(n_1,n_2)$ truncation $\mathcal T_{n_1,n_2}[f]$ keeps the terms $n=0,\ldots,n_1$ of the mixed series and $n=0,\ldots,n_2$ of the denominator series. The index $n_2$ thus counts corrections in higher powers of the denominator fluctuation, and $n_1$ counts higher-order coupling between numerator and denominator. The corresponding variance approximation is
\[
\mu_2^{(n_1,n_2)}:=\mathbf{E}\!\left[\mathcal T_{n_1,n_2}[f]^2\right].
\]
In contrast to the infinite series, the truncations are finite sums and are well defined in all cases. We write $P_{n_1,n_2}$ for the corresponding uncentered partial sum of Eq.~\eqref{eq:p12}, so that $\mathcal T_{n_1,n_2}[f]=P_{n_1,n_2}-\mathbf{E}P_{n_1,n_2}$ and $\mu_2^{(n_1,n_2)}=\mathrm{Var}(P_{n_1,n_2})$. Before bounding the truncation error, we note how the truncated variances behave as more terms are retained.

\begin{remark}[Convergence of the truncated variances]
\label{rem:trunc-conv}
For a network of given finite size there are finitely many reporting patterns, although their number grows exponentially with $N$, and each truncated variance $\mu_2^{(n_1,n_2)}=\mathrm{Var}(P_{n_1,n_2})$ is an average of the squared deviation over all of these patterns. Under MCAR every reporting pattern has positive probability, so the truncated variances converge as more terms are retained in the series precisely when the partial sums $P_{n_1,n_2}$ converge on every pattern. We consider the nonempty patterns and the all-missing event $A_\emptyset$ in turn.

On the nonempty patterns, the series converges if and only if $\alpha>\tfrac12$ (Proposition~\ref{prop:var-conv}), apart from trivial cases in which the average with complete reporting, $\bar r=\sum_i\beta_ir_i$, is identically zero. When $\alpha>\tfrac12$, both sums converge on every such pattern, implying that $P_{n_1,n_2}$ approaches $\widehat r$ whenever both indices increase, whatever their difference. When $\alpha\le\tfrac12$, the series diverges on the pattern in which every site reports, and the sequence of truncated variances does not approach the exact variance.

On $A_\emptyset$, by contrast, the series does not converge unless $m_r=0$, but each partial sum is finite with $U=-m_R$ and $V=-m_S$, so that every term of the mixed series equals $-m_r$ and every term of the denominator series equals $+m_r$, where $m_r=m_R/m_S$ is the weighted mean of the site means. Since $P_{n_1,n_2}$ contains $n_1+1$ terms of the first kind and $n_2+1$ of the second,
\[
P_{n_1,n_2}=(n_2-n_1)\,m_r\qquad\text{on }A_\emptyset ,
\]
and depends only on the difference $n_2-n_1$, remaining fixed if the difference is held fixed, and equaling $m_r$ at every order along the sequence $n_2=n_1+1$. In contrast it grows without bound if the difference grows, as in the $(0,n_2)$ family, unless $m_r=0$.

Combining the two cases, we see that the truncated variances at fixed $N$ converge if $\alpha>\tfrac12$ and $n_2-n_1$ is held fixed, and otherwise fail to converge, apart from trivial cases (Table~\ref{tab:convergence}). Their limit is the exact variance of $\widehat r$ computed with the value $(n_2-n_1)\,m_r$ on $A_\emptyset$. This value is implied by the sequence of truncations and need not equal the convention $r_*$ chosen in the Introduction, but since $A_\emptyset$ has probability $(1-\alpha)^N$ the difference is exponentially small in $N$. Along $n_2=n_1+1$ the implied value is $m_r$, a natural choice and the one used for the exact uniform-weight variance in Section~\ref{sec:exact-uniform}.

Convergence at fixed $N$ therefore depends on how the truncations are increased, and holds only for $\alpha>\tfrac12$. The difference $n_2-n_1$ also matters on every pattern for another reason: only truncations with $n_2=n_1+1$ are invariant under a common shift of the field (Section~\ref{sec:shift}). By contrast, the bounds on the truncation error and the fixed-order accuracy established below hold for every fixed $\alpha\in(0,1)$, whether or not the series converges.
\end{remark}

\begin{Lemma}[Size of the terms at fixed availability]
\label{lem:term-order}
(i) For $n\ge1$, the $n$th centered terms of the mixed and denominator series in Eq.~\eqref{eq:sec5-center} satisfy
\[
\Big\lVert \frac{UV^n-\mathbf{E}(UV^n)}{m_S^{\,n+1}}\Big\rVert_2
\le\frac{\lVert U\rVert_4}{m_S}\,\rho_{4n}^{\,n},
\qquad
\Big\lVert \frac{m_R\,(V^n-\mathbf{E}V^n)}{m_S^{\,n+1}}\Big\rVert_2
\le\frac{|m_R|}{m_S}\,\rho_{2n}^{\,n}.
\]
(ii) Under MCAR with a fixed availability $0<\alpha<1$, for every $p\ge2$ there is a constant $C_{p,\alpha}$, independent of $N$ and of the weights, such that
\[
\rho\le\rho_p\le C_{p,\alpha}\,\rho .
\]
(iii) Consequently, at fixed $\alpha$, as $\rho\to0$ through increasingly diffuse weights, and in particular for uniform weights as $N\to\infty$, the $n$th denominator term is $\mathcal O(\rho^n)$, and so is the $n$th mixed term whenever $\lVert U\rVert_4/m_S=\mathcal O(1)$, as it is for bounded fields.
\end{Lemma}

The proof is given in Appendix~\ref{app:proofs-sec2}. Part (ii) holds because $V$ is a weighted sum of many small, independent, centered reporting indicators. Such a sum has a typical size set by its standard deviation, and its higher moments are, up to constants that depend on $\alpha$, powers of that standard deviation, as for a Gaussian variable.

By part (iii), the $n$th denominator term satisfies $\lVert m_R(V^n-\mathbf{E}V^n)\rVert_2/m_S^{\,n+1}=\mathcal O(\rho^n)$, so each additional order of the denominator series contributes one further factor of $\rho$. The mixed terms are smaller still when the field is only locally dependent (Eq.~\eqref{eq:local-dep-sec2} below). Since $V$ is independent of the field, $\mathbf{E}[U^2V^{2n}]$ involves the field only through its means and covariances, and the $n$th mixed term is then $\mathcal O(\rho^{n+1})$ (Appendix~\ref{app:proof-cor-mar} gives the case $n=1$). For the $(0,n_2)$ truncation, the first omitted denominator term is therefore $\mathcal O(\rho^{n_2+1})$, while the first omitted mixed term, $UV/m_S^2$, is $\mathcal O(\rho^2)$ whatever the value of $n_2$. The larger of the two limits the error, which is the origin of the rate $\rho^{\min(2,n_2+1)}$ in Corollary~\ref{cor:mar-error}. These orders are used below for the large-$N$ rates, although the finite-sample bounds retain the higher moments $\rho_p$ explicitly and do not require part (ii).

A related counting rule is used repeatedly below. For averages of $N$ independent centered quantities with bounded moments, the expectation of a product of $k$ such averages is of order at most $N^{-\lceil k/2\rceil}$, since a term of the expanded product has nonzero expectation only when every site index in it appears at least twice. Products of two first-order fluctuations are therefore of order $N^{-1}$, products of three or four of order $N^{-2}$, and products of five or six of order $N^{-3}$. Under the local-dependence condition, the same ordering applies to products in which the numerator fluctuation appears at most twice, as in all the variance formulas below.

\medskip
The bound rests on an elementary idea. The terms of Eq.~\eqref{eq:p12} carry successive powers of the normalized denominator fluctuation $V/m_S$, whose typical size is $\rho$. If this fluctuation were bounded, say $|V/m_S|\le\tau<1$ on every reporting pattern, the omitted terms of a truncation would form a geometric tail, at most the first omitted term divided by $1-\tau$. Truncating the denominator series after $n_2$ terms would then leave a remainder no larger than
\begin{equation}
\left|
\underbrace{
\frac{m_R}{m_S}\sum_{n>n_2}(-1)^n(V/m_S)^n
}_{\text{omitted denominator tail}}
\right|
\le
\frac{|m_R|}{m_S}
\frac{|V/m_S|^{n_2+1}}{1-\tau},
\label{eq:elementary-tail}
\end{equation}
and, since the $(0,n_2)$ truncation keeps only the $n=0$ mixed term, its omitted mixed tail would be at most $|U||V|/\{m_S^2(1-\tau)\}$. Squaring and averaging would bound the remainder through the mixed moment $\mathbf{E}[U^2V^2]$ and the denominator moment $\rho_{2(n_2+1)}^{\,2(n_2+1)}$.

Under MCAR, however, $|V/m_S|$ is not smaller than one on every reporting pattern. When unusually few sites report, or, for $\alpha<\tfrac12$, unusually many, $S$ departs substantially from its mean and the geometric bound fails. Such patterns are rare when $\rho$ is small. We therefore apply the geometric bound on a good event $G$, where $S$ is close to its mean, and on its complement $G^c$ use only that the field is bounded, so that the rare patterns contribute only through their probability.

\begin{Theorem}[Truncation error controlled by denominator concentration]
\label{thm:rho-error}
Suppose that $|r_i|\le\Lambda$ for all $i$, that $m_S>0$, and that $f=r_*$ on $A_\emptyset$ with $|r_*|\le\Lambda$, so that $|f|\le\Lambda$ everywhere.

Let $P=P_{0,n_2}$ denote the uncentered $(0,n_2)$ partial sum of
Eq.~\eqref{eq:p12}, so that
\[
\mu_2^{(0,n_2)}=\mathrm{Var}(P).
\]
Fix $\tau\in(0,1)$ and define the good event
\[
G:=\{|S-m_S|\le\tau m_S\},
\]
together with the good-event remainder
\[
\delta_G
:=
\left\{\mathbf{E}\big[(f-P)^2\mathbf{1}_G\big]\right\}^{1/2}.
\]
Then
\begin{equation}
\big|\mathrm{Var}(f)-\mu_2^{(0,n_2)}\big|
\le
\underbrace{
2\lVert f-\mathbf{E}f\rVert_2\,\delta_G+\delta_G^2
}_{\text{good-event truncation contribution}}
+
\underbrace{
C_\Lambda\,\mathbf{P}(G^c)^{1/2}
}_{\text{rare-event contribution}},
\label{eq:thm-var-to-delta}
\end{equation}
where
\[
C_\Lambda
:=
4\Lambda\lVert f-P\rVert_2
+
\left\{\mathbf{E}(f-P)^4\right\}^{1/2}.
\]

$C_\Lambda$ is finite for fixed $n_2$ and $m_S>0$, and under MCAR with fixed $\alpha$ and a bounded field it is bounded uniformly in $N$.

The good-event remainder satisfies
\begin{equation}
\delta_G^2
\le
\frac{2}{m_S^2(1-\tau)^2}
\left[
\underbrace{
\frac{\mathbf{E}[U^2V^2]}{m_S^2}
}_{\text{coupling term}}
+
\underbrace{
m_R^2\rho_{2(n_2+1)}^{\,2(n_2+1)}
}_{\text{pure-denominator tail}}
\right].
\label{eq:thm-delta}
\end{equation}
By Chebyshev's inequality, $\mathbf{P}(G^c)\le\rho^2/\tau^2$, and under MCAR, Bernstein's inequality gives the exponential bound
\[
\mathbf{P}(G^c)
\le
2\exp\left[
-\frac{\tau^2m_S^2/2}
{\mathrm{Var}(S)+\tfrac13\tau m_S\max_i\beta_i}
\right].
\]
\end{Theorem}

The proof is given in Appendix~\ref{app:proof-thm}. The theorem holds for every network size and makes no limit statement. The all-missing event lies in the rare event, $A_\emptyset\subseteq G^c$, since $S=0$ gives $|S-m_S|=m_S>\tau m_S$, so the convention used on $A_\emptyset$ influences the bound only through the rare-event term.

Under MCAR at fixed $\alpha$, Lemma~\ref{lem:term-order}(ii) makes every $\rho_p$ comparable to $\rho$, so that $\rho$ becomes the single small parameter. For increasingly diffuse weights, $W_2\to0$ and hence $\rho\to0$, and for uniform weights the rare-event contribution is moreover exponentially small in $N$, so that the good-event term determines the order of the error. Corollary~\ref{cor:mar-error} states this for uniform weights. Boundedness of the field is used mainly on the rare event $G^c$, and could be relaxed to moment conditions on the field, which we do not pursue here.

\begin{Corollary}[MCAR, uniform weights, and a bounded field]
\label{cor:mar-error}
Suppose reporting is MCAR with fixed availability $0<\alpha<1$, the weights are uniform, $\beta_i=1/N$, and $|r_i|\le \Lambda$. Then, for every fixed $n_2\ge0$,
\[
\big|\mathrm{Var}(f)-\mu_2^{(0,n_2)}\big|
\longrightarrow0
\qquad\text{as }N\to\infty,
\]
without any restriction on the spatial dependence of the field.

If, in addition,
\begin{equation}
\sum_{i\ne j}|\sigma_{ij}|=\mathcal O(N),
\label{eq:local-dep-sec2}
\end{equation}
that is, correlations between sites decay quickly enough with distance that the total covariance grows no faster than the number of sites, then the good-event remainder satisfies
\[
\delta_G
=
\mathcal O\!\left(
\rho^{\min(2,n_2+1)}
\right)
=
\begin{cases}
\mathcal O(\rho)=\mathcal O(N^{-1/2}),&n_2=0,\\[2pt]
\mathcal O(\rho^2)=\mathcal O(N^{-1}),&n_2\ge1.
\end{cases}
\]
The rare-event contribution is exponentially small:
\[
C_\Lambda\,\mathbf{P}(G^c)^{1/2}
=
\mathcal O(\exp(-cN/2))
\]
for some $c=c(\alpha,\tau)>0$. The pure-denominator contribution to $\delta_G^2$ is
\[
\mathcal O\!\left(\rho^{2(n_2+1)}\right).
\]
The corollary holds for every fixed $\alpha\in(0,1)$ and no condition $\alpha>\tfrac12$ is needed here.
\end{Corollary}

The proof of Corollary~\ref{cor:mar-error} is given in Appendix~\ref{app:proof-cor-mar}.

Corollary~\ref{cor:mar-error} concerns fixed $\alpha$, uniform weights, and large $N$. In this regime $\rho^2=(1-\alpha)/(\alpha N)$ and $\mathbf{P}(G^c)=\mathcal O(\exp(-cN))$, so the denominator concentrates around its mean and, for every fixed $n_2$, the truncation error of the $(0,n_2)$ approximation vanishes as $N\to\infty$. This is a statement about the absolute error. Under local dependence the variance itself decreases as $N^{-1}$, so the relative error need not vanish: that of $(0,0)$ tends to a nonzero limit when $m_r\neq0$ (Section~\ref{sec:numerics}), whereas for $n_2\ge1$ it also vanishes. These statements concern $N\to\infty$ with $\alpha$ fixed, and do not assert uniformity when $\alpha$ varies with $N$.

The rates in the corollary follow from the sizes of the terms noted after Lemma~\ref{lem:term-order}. For $n_2\ge1$ the bound on $\delta_G$ does not improve further with $n_2$, because it is then limited by the omitted mixed terms, which additional denominator corrections do not reduce, and improvement in accuracy strictly requires increasing $n_1$ together with $n_2$ (Remark~\ref{rem:translation}). Together with Theorem~\ref{thm:rho-error} and $\lVert f-\mathbf{E}f\rVert_2=\mathcal O(\rho)$ under local dependence, the corollary gives variance-error bounds of order $\rho^2$ for $(0,0)$ and $\rho^3$ for $(0,n_2)$ with $n_2\ge1$; Proposition~\ref{prop:rho4} sharpens the latter to $\rho^4$ for $(0,1)$.

For $(0,1)$, the algebra of the truncation gives a sharper bound than these general norm estimates.

\begin{Proposition}[Sharper rate for the $(0,1)$ approximation]
\label{prop:rho4}
Suppose reporting is MCAR with uniform weights $\beta_i=1/N$, fixed
$\alpha\in(0,1)$, and a uniformly bounded field, $|r_i|\le \Lambda$. Assume
also the local-dependence condition Eq.~\eqref{eq:local-dep-sec2}. Then
\[
\rho^2=\frac{1-\alpha}{\alpha N}
\]
and, as $N\to\infty$, the $(0,1)$ variance approximation satisfies
\[
\big|\mathrm{Var}(f)-\mu_2^{(0,1)}\big|
=
\mathcal O(\rho^4)
=
\mathcal O(N^{-2}).
\]
Under the same assumptions, the $(0,0)$ approximation satisfies
\[
\big|\mathrm{Var}(f)-\mu_2^{(0,0)}\big|
=
\mathcal O(\rho^2)
=
\mathcal O(N^{-1}).
\]
The implied constants are independent of $N$, but may depend on $\Lambda$, $\alpha$, and the constant in Eq.~\eqref{eq:local-dep-sec2}.
\end{Proposition}

The proof of Proposition~\ref{prop:rho4} is given in Appendix~\ref{app:proof-rho4}. Thus, at the level of these upper bounds, retaining the first denominator correction improves the order of the truncation error in variance from $\mathcal O(\rho^2)$ to $\mathcal O(\rho^4)$, equivalently from $\mathcal O(N^{-1})$ to $\mathcal O(N^{-2})$, which helps explain the good accuracy of the estimator of \cite{Seshadri2018}. The reason for the improvement is visible in an exact identity. Wherever $S>0$,
\[
f-P_{0,1}=\frac{-UV/m_S^2+m_RV^2/m_S^3}{1+V/m_S}.
\]
The remainder is therefore of second order in the fluctuations, whereas $P_{0,1}$ itself is of first order. The error in the variance involves the covariance of the two, an average of a product of three fluctuations, and the variance of the remainder, an average of a product of four, and by the counting rule both are of order $N^{-2}$. For $(0,0)$ the remainder contains the first-order term $-m_RV/m_S^2$, and unless $m_R=0$ the error is of order $N^{-1}$.

Condition~\eqref{eq:local-dep-sec2} is sufficient but not necessary for the rate $\rho^4$. Under persistent spatial dependence, where $\sum_{i\ne j}|\sigma_{ij}|$ grows faster than $N$, the error of $(0,1)$ is instead of order $\rho^2$ for uniform weights (Proposition~\ref{prop:uniform-decomp}), although Corollary~\ref{cor:mar-error}, which requires only a bounded field, still holds.

\begin{remark}[Heterogeneous independent reporting]
\label{prop:het-alpha}
Suppose the reporting indicators are mutually independent and independent of the field, with $s_i\sim\mathrm{Bernoulli}(\alpha_i)$, where the site-specific reporting probabilities $\alpha_i$ need not be equal, and $m_S=\sum_i\beta_i\alpha_i>0$. Then $\mathrm{Var}(S)=\sum_i\beta_i^2\alpha_i(1-\alpha_i)$, and
\begin{equation}
\rho^2
=
\frac{\sum_i\beta_i^2\alpha_i(1-\alpha_i)}{\left(\sum_i\beta_i\alpha_i\right)^2},
\label{eq:rho-het}
\end{equation}
which reduces to $\rho^2=\tfrac{1-\alpha}{\alpha}W_2$ when $\alpha_i\equiv\alpha$. Theorem~\ref{thm:rho-error} holds with these values of $m_S$ and $\mathrm{Var}(S)$; in particular, $\mathbf{P}(G^c)\le\rho^2/\tau^2$, and the Bernstein bound applies with $X_i:=\beta_i(s_i-\alpha_i)$. If the reporting probabilities depend on observed covariates and, conditional on those covariates, the reporting indicators are independent of one another and of the unobserved field values, the same formulas and bounds apply with all probabilities and moments interpreted conditionally. This gives the natural missing-at-random analogue of the MCAR setting.
\end{remark}

Table~\ref{tab:three-results} summarizes the three results of this section and the intuition behind each. Section~\ref{sec:truncated} derives the low-order approximations explicitly under MCAR, and Section~\ref{sec:numerics} compares them with exact calculations.

\begin{table}[t]
\centering
\small
\caption{The three questions of Section~\ref{sec:moments}: the result for each, and the intuition behind it.}
\label{tab:three-results}
\renewcommand{\arraystretch}{1.3}
\begin{tabular}{>{\raggedright\arraybackslash}p{0.15\linewidth} >{\raggedright\arraybackslash}p{0.29\linewidth} >{\raggedright\arraybackslash}p{0.29\linewidth} >{\raggedright\arraybackslash}p{0.16\linewidth}}
\hline
Question & Result & Intuition & Where \\
\hline
Convergence: more terms, fixed $N$ &
Under MCAR, the truncated variances converge only if $\alpha>\tfrac12$ and $n_2-n_1$ is held fixed, with the limit being the exact variance with the value $(n_2-n_1)m_r$ on $A_\emptyset$ &
The ratio is a geometric series in $x=V/m_S$, which needs $|x|<1$ on every pattern, and $|x|=(1-\alpha)/\alpha$ when all sites report; on $A_\emptyset$ the partial sums equal $(n_2-n_1)m_r$ &
Proposition~\ref{prop:var-conv}, Remark~\ref{rem:trunc-conv}, Table~\ref{tab:convergence} \\
Truncation error: given truncation, given $N$ &
For $(0,n_2)$, bounded at every $N$ through $\rho$, the higher moments $\rho_p$, and $\mathbf{P}(G^c)\le\rho^2/\tau^2$, which is exponentially small under MCAR with diffuse weights, whether or not the series converges &
Where $S$ is near its mean, the omitted terms form a geometric tail; the remaining patterns are rare and contribute only through their probability &
Theorem~\ref{thm:rho-error} \\
Fixed-order accuracy: given truncation, $N\to\infty$ &
For uniform weights and a bounded field, every $(0,n_2)$ error vanishes at any fixed $\alpha$; under local dependence it is $\mathcal O(\rho^2)$ for $(0,0)$ and $\mathcal O(\rho^4)$ for $(0,1)$ &
Each denominator order adds a factor $\rho$, and a mixed term one more under local dependence; $(0,0)$ omits the first-order term $-m_RV/m_S^2$, which $(0,1)$ keeps &
Lemma~\ref{lem:term-order}, Corollary~\ref{cor:mar-error}, Proposition~\ref{prop:rho4} \\
\hline
\end{tabular}
\end{table}

\section{Truncated variance approximations and their properties}
\label{sec:truncated}

Section~\ref{sec:moments} defined the truncations $\mu_2^{(n_1,n_2)}$ through Eq.~\eqref{eq:sec5-center} and bounded the truncation error of the $(0,n_2)$ family. This section establishes explicit formulas for the low-order truncations under MCAR, identifies the shift-invariant truncations as those free of a spurious dependence on the mean level, derives the exact variance for uniform weights and its decomposition, and explains why higher truncations need not improve on the first-order approximation, with Table~\ref{tab:truncation-properties} summarizing the main results.

The formulas contain three kinds of contribution: (i) numerator variability, through $\mathrm{Var}(R)$; (ii) the leading effect of a random denominator, through $\mathrm{Var}(S)$ and $\mathrm{Cov}(R,S)$; and (iii) higher-order numerator-denominator terms, through $\mu_{2,1}=\mathbf{E}[U^2V]$, $\mu_{1,2}=\mathbf{E}[UV^2]$, and $\kappa=\mathrm{Var}(UV)$. Under MCAR each is a weighted combination of the first and second moments of the field.

% ============================================================
\subsection{Variance formulas for the truncations: $(0,0)$, $(0,1)$, $(1,0)$, $(1,1)$}
\label{sec:truncated-variance}

The low-order truncations are sums of the building blocks
\[
A := \frac{U}{m_S},\qquad
B := -\,\frac{m_R}{m_S^2}\,V,\qquad
D := -\,\frac{UV-\mathbf{E}[UV]}{m_S^2},
\]
where term $A$ is the fluctuation obtained if the denominator is held at its mean, and therefore represents the ordinary variability of the weighted numerator, $B$ is the first correction for a random denominator, and $D$ is the leading coupling between numerator and denominator. Supplementary Information Section~S1 details the algebra, and the resulting first four truncations are
\begin{align}
\mu_2^{(0,0)} &= \mathbf{E}A^2
 \;=\; \frac{\mathrm{Var}(R)}{m_S^2}, \label{eq:var00}\\
\mu_2^{(0,1)} &= \mathbf{E}(A+B)^2
 \;=\; \frac{\mathrm{Var}(R)}{m_S^2}
      + \frac{m_R^2}{m_S^4}\,\mathrm{Var}(S)
      - \frac{2m_R}{m_S^3}\,\mathrm{Cov}(R,S), \label{eq:var01}\\
\mu_2^{(1,0)} &= \mathbf{E}(A+D)^2
 \;=\; \frac{\mathrm{Var}(R)}{m_S^2}
      - \frac{2}{m_S^3}\,\mu_{2,1}
      + \frac{1}{m_S^4}\,\kappa,
\label{eq:var10}\\
\mu_2^{(1,1)} &= \mathbf{E}(A+B+D)^2 \nonumber\\
&= \frac{\mathrm{Var}(R)}{m_S^2}
 + \frac{m_R^2}{m_S^4}\,\mathrm{Var}(S)
 - \frac{2m_R}{m_S^3}\,\mathrm{Cov}(R,S)
 - \frac{2}{m_S^3}\,\mu_{2,1}
 + \frac{2m_R}{m_S^4}\,\mu_{1,2}
 + \frac{1}{m_S^4}\,\kappa,
\label{eq:var11}
\end{align}
where
\[
\mu_{2,1}:=\mathbf{E}[U^2V],\qquad
\mu_{1,2}:=\mathbf{E}[UV^2],\qquad
\kappa:=\mathrm{Var}(UV)
=\mathbf{E}[U^2V^2]-\big(\mathbf{E}[UV]\big)^2
=\mathbf{E}[U^2V^2]-\big(\mathrm{Cov}(R,S)\big)^2.
\]

% ============================================================
\subsection{Variance approximations under MCAR}
\label{sec:mar-variance}

Under MCAR,
$s_i\overset{\mathrm{i.i.d.}}{\sim}\mathrm{Bernoulli}(\alpha)$,
independent of $\{r_i\}$, and $\sum_i\beta_i=1$.

Let
\[
m_i:=\mathbf{E}r_i,\qquad
\sigma_i^2:=\mathrm{Var}(r_i),\qquad
\sigma_{ij}:=\mathrm{Cov}(r_i,r_j),\qquad
W_2:=\sum_i \beta_i^2,
\]
and define the weighted mean and the mean-weight moments
\[
m_r:=\sum_i\beta_i m_i,\qquad
M_2:=\sum_i\beta_i^2 m_i,\qquad
M_k:=\sum_i\beta_i^k m_i.
\]
Here $m_r$ is the weighted spatial mean of the site means, whereas $m_R=\mathbf{E}R=\alpha m_r$ is the mean of the reported numerator. Supplementary Information Section~S1 derives the basic moments under MCAR
\begin{equation}
m_S=\alpha,\qquad
\mathrm{Var}(S)=\alpha(1-\alpha)W_2,\qquad
m_R=\alpha\,m_r,\qquad
\mathrm{Cov}(R,S)=\alpha(1-\alpha)M_2.
\label{eq:sec5-MAR-basic}
\end{equation}
The variance of the numerator then has three contributions, from the variance of the field, its covariance, and the mean, each affected by the availability $\alpha$. These formulas reproduce the results of \cite{Seshadri2018} (in the notation of that paper, $\mathrm{Var}(R)\leftrightarrow\sigma_R^2$, $\mathrm{Var}(S)\leftrightarrow\sigma_S^2$, and $\mathrm{Cov}(R,S)\leftrightarrow\sigma_{RS}^2$), and their development is detailed in Supplementary Information Section~S1.
\begin{equation}
\mathrm{Var}(R)
=\alpha\sum_i\beta_i^2\sigma_i^2
+\alpha^2\!\!\sum_{i\ne j}\beta_i\beta_j\sigma_{ij}
+\alpha(1-\alpha)\sum_i\beta_i^2 m_i^2.
\label{eq:sec5-VarR}
\end{equation}

\paragraph{(0,0) and (0,1) under MCAR.}
Substituting Eqs.~\eqref{eq:sec5-MAR-basic}-\eqref{eq:sec5-VarR} into Eqs.~\eqref{eq:var00}--\eqref{eq:var01} yields
\begin{align}
\mu_2^{(0,0)}
&= \frac{1}{\alpha}\sum_i\beta_i^2\sigma_i^2
+ \sum_{i\ne j}\beta_i\beta_j\sigma_{ij}
+ \frac{1-\alpha}{\alpha}\sum_i\beta_i^2 m_i^2,
\label{eq:mar00}\\[0.35em]
\mu_2^{(0,1)}
&= \mu_2^{(0,0)}
+ \frac{1-\alpha}{\alpha}\Big(m_r^{\,2}W_2-2m_rM_2\Big)
\label{eq:mar01a}\\
&= \frac{1}{\alpha}\sum_i\beta_i^2\sigma_i^2
+ \sum_{i\ne j}\beta_i\beta_j\sigma_{ij}
+ \frac{1-\alpha}{\alpha}\Big(\sum_i\beta_i^2 m_i^2+m_r^{\,2}W_2-2m_rM_2\Big).
\label{eq:mar01}
\end{align}
The $(0,0)$ approximation contains the intrinsic field variance and covariance together with a contribution from the site means. Under MCAR, the diagonal variance term is inflated by the factor $1/\alpha$, whereas the off-diagonal spatial-covariance term is not. A single site enters the numerator only when it reports, with probability $\alpha$, whereas a pair of sites enters only when both report, with probability $\alpha^2$, and normalization by $(\mathbf{E}S)^2=\alpha^2$ cancels the latter but not the former. The additional mean term arises from uncertainty about which sites report. The $(0,1)$ approximation adds the leading linear denominator-variability term
$m_R^2\mathrm{Var}(S)/m_S^4$ and the leading covariance correction $-2m_R\mathrm{Cov}(R,S)/m_S^3$.

The two leading approximations have a simpler form that shows their structure directly. Since $\sum_i\beta_i^2m_i^2+m_r^2W_2-2m_rM_2=\sum_i\beta_i^2(m_i-m_r)^2$,
\begin{align}
\mu_2^{(0,0)} &= \sum_{i,j}\beta_i\beta_j\sigma_{ij}
+\frac{1-\alpha}{\alpha}\sum_i\beta_i^2\left(\sigma_i^2+m_i^2\right),
\label{eq:mar00-compact}\\
\mu_2^{(0,1)} &= \sum_{i,j}\beta_i\beta_j\sigma_{ij}
+\frac{1-\alpha}{\alpha}\sum_i\beta_i^2\left\{\sigma_i^2+(m_i-m_r)^2\right\},
\label{eq:mar01-compact}
\end{align}
with $\sigma_{ii}=\sigma_i^2$. The first term is $\mathrm{Var}(\bar r)$, the variance of the average with complete reporting, $\bar r=\sum_i\beta_ir_i$, and the second approximates the contribution of missing data, $\mathrm{Var}(\widehat r)-\mathrm{Var}(\bar r)$, which grows in proportion to $(1-\alpha)/\alpha$. The two approximations differ only in the reference point for the second moment at each site: $(0,0)$ measures departures of the site values from zero, whereas $(0,1)$ measures departures from the weighted spatial mean $m_r$. This is why a common shift of the field changes $\mu_2^{(0,0)}$ but not $\mu_2^{(0,1)}$, and why the mean-induced term vanishes at $(0,1)$ when all site means are equal. Equivalently, $(0,1)$ is $(0,0)$ applied to the field after its weighted mean $m_r$ has been subtracted from every site. Since the exact variance is unchanged by this common shift, centering the field in this way is the simplest route to a shift-invariant approximation. Equation~\eqref{eq:mar01-compact} is Eq.~(25) of \cite{Seshadri2018}, written there in matrix form, in Eqs.~(26)--(27), as $\beta^{\mathsf T}C\beta$ with $C=S_r+\tfrac{1-\alpha}{\alpha}\left(S_r^{d}+F_r\right)$.

The difference between the two matters when the mean level is large compared with the variability of the field. In a small example with $N=4$ i.i.d.\ sites, $m=2\sigma$, and $\alpha=0.5$ (Supplementary Information Section~S1), $(0,1)$ is within about $7\%$ of the exact variance, whereas $(0,0)$ is nearly three times too large, even though $\rho^2=0.25$ is not small.

\paragraph{Mixed moments under MCAR.}
Under MCAR, the mixed moments $\mu_{1,2}$, $\mu_{2,1}$, and $\kappa$ that complete the formulas for $(1,0)$ and $(1,1)$ also have closed forms involving the first and second moments of the field. These formulas can be derived by separating the reporting indicators from the field and counting the distinct sites in each product, and are derived in Supplementary Information Section~S1; for example, $\mu_{1,2}=\alpha(1-\alpha)(1-2\alpha)\sum_i\beta_i^3m_i$. For uniform weights and under the local-dependence condition Eq.~\eqref{eq:local-dep-sec2}, all three are of order $N^{-2}$, and why the corrections they add need not improve on $(0,1)$ is discussed in Remark~\ref{rem:hierarchy}. Supplementary Information Table~S2 collects the low-order approximations together with their MCAR specializations.

\subsection{Invariance under a common shift}
\label{sec:shift}

The behavior of $(0,0)$ and $(0,1)$ under a common shift of the field, noted after Eq.~\eqref{eq:mar01-compact}, is a special case of a general property that determines which truncations are consistent with the exact variance.

\begin{remark}[Shift-invariant truncations]
\label{rem:translation}
Adding a constant $b$ to every field value ($r_i\mapsto r_i+b$, with $r_*\mapsto r_*+b$) shifts $\widehat r$ by exactly $b$, so $\mathrm{Var}(\widehat r)$ is unchanged. Under this shift $U\mapsto U+bV$ and $m_R\mapsto m_R+b\,m_S$. Writing $x=V/m_S$, the mixed series therefore gains $-b\sum_{j=1}^{n_1+1}(-x)^j$ and the denominator series gains $b\sum_{j=0}^{n_2}(-x)^j$. The two additions sum to the constant $b$, which centering removes, if and only if $n_2=n_1+1$. Counting $U$ and $V$ as first-order fluctuations, the $n$th mixed term is of order $n+1$ and the $n$th denominator term of order $n$. The truncations $(0,1)$, $(1,2)$, $(2,3),\ldots$ are therefore exactly those that keep both families to the same total order $n_1+1$, as in a classical multivariate delta-method expansion, and we refer to these as the shift-invariant truncations. All other truncations, including $(0,0)$, $(1,0)$, $(1,1)$, and $(0,2)$, change under the shift by amounts that grow quadratically in $b$, equivalently in the mean level. For uniform weights, fixed $\alpha$, and a bounded field, these additional terms are of order $m_r^2N^{-1}$ for $(0,0)$ and $(1,0)$, and of order $m_r^2N^{-2}$ for $(1,1)$ and $(0,2)$. Invariance under the shift does not by itself make a truncation more accurate (Fig.~\ref{fig:translation}).

The shift-invariant truncations have two further properties. On the all-missing event the partial sums $P_{n_1,n_1+1}$ take the value $m_r$ at every order, so they correspond to the convention $r_*=m_r$ used in Eq.~\eqref{eq:exact-uniform-var}, and for $\alpha>\tfrac12$ they converge at fixed $N$ (Remark~\ref{rem:trunc-conv}). Along this sequence the order of the variance error improves only at every second step. Heuristically, the error of the truncation of total order $K=n_1+1$ is governed by averages of products of $K+2$ centered fluctuations, which are of order $N^{-\lceil (K+2)/2\rceil}$ under local dependence (the counting rule of Section~\ref{sec:accuracy}). For the homogeneous benchmark of Section~\ref{sec:numerics} at $\alpha=0.3$, exact calculations (Supplementary Information Table~S3) give errors that decrease as $N^{-2}$ for $(0,1)$ and $(1,2)$, as $N^{-3}$ for $(2,3)$ and $(3,4)$, and as $N^{-4}$ for $(4,5)$.
\end{remark}

% ============================================================
\subsection{Further extension of the hierarchy: the \texorpdfstring{$(0,2)$}{(0,2)} truncation}
\label{sec:02-20}

Increasing $n_2$ adds terms in higher powers of the denominator fluctuation, which carry the mean level $m_R$, whereas increasing $n_1$ adds terms in which numerator fluctuations interact with higher powers of the denominator fluctuation. The $(0,2)$ truncation illustrates the first of these effects.

The $(0,2)$ truncation adds the next denominator-only term, $\mathcal T_{0,2}[f]=A+B+B_2$ with
\[
B_2 := \frac{m_R}{m_S^3}\big(V^2-\mathbf{E}V^2\big),
\]
leading to
\begin{align}
\mu_2^{(0,2)}
&= \mu_2^{(0,1)}
 \;+\; \frac{2m_R}{m_S^4}\,\mu_{1,2}
 \;-\; \frac{2m_R^2}{m_S^5}\,\mathrm{Cov}(V,V^2)
 \;+\; \frac{m_R^2}{m_S^6}\,\mathrm{Var}(V^2).
\label{eq:var02-general}
\end{align}
Under MCAR with uniform weights, the $\mu_{1,2}$ and $\mathrm{Cov}(V,V^2)$ contributions cancel exactly (Supplementary Information Section~S1), and Eq.~\eqref{eq:var02-general} reduces to
\begin{equation}
\mu_2^{(0,2)}=\mu_2^{(0,1)}+\frac{m_r^2}{\alpha^4}\,\mathrm{Var}(V^2)
\qquad(\beta_i=1/N).
\label{eq:var02-uniform}
\end{equation}

Under MCAR with uniform weights, $\mathrm{Var}(V^2)=\mathcal O(W_2^2)=\mathcal O(N^{-2})$, whereas under local dependence the leading terms of $(0,0)$ and $(0,1)$ are of order $N^{-1}$. The $(0,2)$ correction is therefore appreciable mainly when $N$ or $\alpha$ is small, and whether it improves the approximation depends on the mean level. The surviving term $m_r^2\mathrm{Var}(V^2)/\alpha^4$ grows as $m_r^2$, whereas the exact variance is unchanged by a common shift of the field. It is therefore an artifact of truncation rather than a physical effect of the mean level, and a large mean level relative to the variability of the field degrades $(0,2)$ (Remark~\ref{rem:translation}).

% ============================================================
\subsection{Uniform weights: exact variance, and what \texorpdfstring{$(0,1)$}{(0,1)} omits}
\label{sec:exact-uniform}

With uniform weights, $\beta_i=1/N$, the weighted mean $m_r$ is the average of the site means, and the formulas are expressed through the spread of the site means,
\[
\mathrm{Var}_m:=\frac{1}{N}\sum_{i=1}^N (m_i-m_r)^2,
\]
together with the sums
\[
\Sigma_\sigma:=\sum_{i=1}^N \sigma_i^2,
\qquad
\Sigma_{\mathrm{off}}:=\sum_{i\neq j}\sigma_{ij},
\qquad
\Sigma_m:=\sum_{i=1}^N m_i^2.
\]

With these, and since $\mathrm{Var}(\bar r)=(\Sigma_\sigma+\Sigma_{\mathrm{off}})/N^2$, the expressions in Eqs.~\eqref{eq:mar00-compact}-\eqref{eq:mar01-compact} yield
\[
\mu_2^{(0,0)}=\mathrm{Var}(\bar r)+\rho^2\Big(\frac{\Sigma_\sigma}{N}+\frac{\Sigma_m}{N}\Big),
\qquad
\mu_2^{(0,1)}=\mathrm{Var}(\bar r)+\rho^2\Big(\frac{\Sigma_\sigma}{N}+\mathrm{Var}_m\Big).
\]
Since $\Sigma_m/N-\mathrm{Var}_m=m_r^2$, the two differ by the mean-induced term $\rho^2m_r^2$. Proposition~\ref{prop:uniform-decomp} below shows that the exact variance has the same structure, with $\rho^2$ replaced by $\gamma_N$ and the bracket by the expected spatial variance of a single realization, $\mathbf{E}s^2=\Sigma_\sigma/N+\mathrm{Var}_m-\mathrm{Var}(\bar r)$.

The closed forms for independent and for i.i.d.\ sites, including those of the mixed moments, are collected in Supplementary Information Table~S1.

The MCAR variance can then be evaluated exactly, without Monte Carlo simulation. Let $K=\sum_i s_i\sim\mathrm{Binomial}(N,\alpha)$ be the number of reporting sites. Conditional on $K=k>0$, every subset of $k$ sites has the same probability, and one obtains
\begin{equation}
\mathrm{Var}(\widehat r)
=
\sum_{k=1}^{N}
\binom Nk\alpha^k(1-\alpha)^{N-k}
\left[
\frac{\Sigma_\sigma}{kN}
+
\frac{k-1}{kN(N-1)}\Sigma_{\mathrm{off}}
+
\frac{N-k}{(N-1)k}\mathrm{Var}_m
\right],
\label{eq:exact-uniform-var}
\end{equation}
where on $K=0$ we use $\widehat r=m_r$. Here ``exact'' means that the average over missing data is computed exactly for given field moments; uncertainty in estimating those moments is a separate problem (Section~\ref{sec:discussion}).

Equation~\eqref{eq:exact-uniform-var} can be rearranged to show where the variance comes from, and what the $(0,1)$ approximation misses. Conditional on $K=k$ and on the field, $\widehat r$ is the mean of a sample of $k$ sites drawn without replacement from $N$. It is therefore centered on $\bar r$, and its variance is the finite-population sampling variance $\frac{N-k}{(N-1)k}\,s^2$, where $s^2=N^{-1}\sum_i(r_i-\bar r)^2$ is the spatial variance of the realization. Averaging over $K$ and over the field separates the variability of $\bar r$ from that due to reporting.

\begin{Proposition}[Exact decomposition for uniform weights]
\label{prop:uniform-decomp}
Suppose reporting is MCAR with uniform weights, and that $\widehat r=m_r$ on $A_\emptyset$. Let $\bar r=N^{-1}\sum_ir_i$ be the average with complete reporting, so that $\mathrm{Var}(\bar r)=(\Sigma_\sigma+\Sigma_{\mathrm{off}})/N^2$, and let
\[
\mathbf{E}s^2:=\frac1N\sum_i\mathbf{E}(r_i-\bar r)^2=\frac{\Sigma_\sigma}{N}+\mathrm{Var}_m-\mathrm{Var}(\bar r)
\]
be the expected spatial variance of a single realization of the field. Then
\begin{equation}
\mathrm{Var}(\widehat r)=\big(1-(1-\alpha)^N\big)\mathrm{Var}(\bar r)+\gamma_N\,\mathbf{E}s^2,
\qquad
\gamma_N:=\sum_{k=1}^N\binom Nk\alpha^k(1-\alpha)^{N-k}\frac{N-k}{(N-1)k},
\label{eq:uniform-decomp}
\end{equation}
where $\gamma_N=\rho^2+\mathcal O(N^{-2})$ at fixed $\alpha$. Since $\mu_2^{(0,1)}=\mathrm{Var}(\bar r)+\rho^2\big(\mathbf{E}s^2+\mathrm{Var}(\bar r)\big)$,
\begin{equation}
\mu_2^{(0,1)}-\mathrm{Var}(\widehat r)
=\rho^2\,\mathrm{Var}(\bar r)-(\gamma_N-\rho^2)\,\mathbf{E}s^2+(1-\alpha)^N\mathrm{Var}(\bar r).
\label{eq:01-error-uniform}
\end{equation}
\end{Proposition}

\textbf{Proof.} For each $k$, the bracket in Eq.~\eqref{eq:exact-uniform-var} equals $\mathrm{Var}(\bar r)+\frac{N-k}{(N-1)k}\,\mathbf{E}s^2$, as can be checked by collecting the coefficients of $\Sigma_\sigma$, $\Sigma_{\mathrm{off}}$, and $\mathrm{Var}_m$. Summing over $k\ge1$ gives Eq.~\eqref{eq:uniform-decomp}. The expression for $\mu_2^{(0,1)}$ follows from Eq.~\eqref{eq:mar01-compact} with $\beta_i=1/N$, and the order of $\gamma_N-\rho^2$ follows from the classical expansion $\mathbf{E}[K^{-1}\mathbf 1_{K\ge1}]=1/(\alpha N)+(1-\alpha)/(\alpha^2N^2)+\mathcal O(N^{-3})$ \cite{Stephan1945}; see also the tables of \cite{Grab1954}. \hfill$\square$

\medskip
The first term of Eq.~\eqref{eq:uniform-decomp} is the variance that remains with complete reporting, and the second is the contribution of missing data, which is proportional to the spatial variance of each realization rather than to the site variances. Strong spatial correlation raises the first term but reduces the second, whereas $(0,1)$ uses the site variances and the spread of site means without this reduction. Under the local-dependence condition, $\mathrm{Var}(\bar r)=\mathcal O(N^{-1})$ and both leading terms in Eq.~\eqref{eq:01-error-uniform} are $\mathcal O(N^{-2})$, in agreement with Proposition~\ref{prop:rho4}. Under persistent spatial dependence $\mathrm{Var}(\bar r)$ does not decrease, and the error of $(0,1)$ is of order $\rho^2$ rather than $\rho^4$. For stationary fields, the decomposition gives the following two cases directly.

\begin{Corollary}[Local-range and persistent spatial dependence]
\label{prop:spatial-dependence}
Suppose reporting is MCAR with uniform weights, and that $r_i=Y(d_i)$, where $Y$ is a second-order stationary field with mean $m$ and covariance function $C(h):=\mathrm{Cov}\{Y(x),Y(x+h)\}$; let $\widehat r$ take any fixed finite value on $A_\emptyset$.
(i) If $N^{-1}\sum_{i\ne j}C(d_i-d_j)\to\Gamma_{\rm off}<\infty$, then
\[
\operatorname{Var}(\widehat r)=\frac{1}{N}\left\{\frac{C(0)}{\alpha}+\Gamma_{\rm off}\right\}+o(N^{-1}).
\]
(ii) If instead $N^{-2}\sum_{i\ne j}C(d_i-d_j)\to c_\infty>0$, then
\[
\operatorname{Var}(\widehat r)\longrightarrow c_\infty,
\qquad
\mu_{2}^{(0,1)}\longrightarrow c_\infty,
\quad\text{although}\quad
\rho^2=\frac{1-\alpha}{\alpha N}\longrightarrow0.
\]
\end{Corollary}

\textbf{Proof.} Changing the value of $\widehat r$ on $A_\emptyset$ changes the variance only by $\mathcal O((1-\alpha)^N)$, so we may set $\widehat r=m_r$ there. Here $\Sigma_\sigma=NC(0)$ and $\mathrm{Var}_m=0$, so $\mathrm{Var}(\bar r)=N^{-1}C(0)+N^{-2}\sum_{i\ne j}C(d_i-d_j)$ and $\mathbf{E}s^2=C(0)-\mathrm{Var}(\bar r)$. In case (i), $\mathrm{Var}(\bar r)=N^{-1}\{C(0)+\Gamma_{\rm off}\}+o(N^{-1})$, $\mathbf{E}s^2=C(0)+\mathcal O(N^{-1})$, and $\gamma_N=(1-\alpha)/(\alpha N)+\mathcal O(N^{-2})$, so Eq.~\eqref{eq:uniform-decomp} gives the stated expansion. In case (ii), $\mathrm{Var}(\bar r)\to c_\infty$ while $\gamma_N\to0$ and $0\le\mathbf{E}s^2\le C(0)$, so $\mathrm{Var}(\widehat r)\to c_\infty$; the same holds for $\mu_2^{(0,1)}=\mathrm{Var}(\bar r)+\rho^2\big(\mathbf{E}s^2+\mathrm{Var}(\bar r)\big)$. \hfill$\square$

\medskip
Corollary~\ref{prop:spatial-dependence} separates two effects: $\rho$ controls the concentration of the denominator, whereas the growth of the covariance sum determines whether the variance of the spatial average decreases as the network grows. A network can therefore have small $\rho$ and still have a variance that does not decrease with $N$, as in Fig.~\ref{fig:stress}(a).

\begin{remark}[A diagnostic for correlated fields]
\label{rem:lambda}
Let $\lambda:=\mathrm{Var}(\bar r)/\mathbf{E}s^2$, the variance of the average with complete reporting relative to the expected spatial variance of a single realization. Since $\gamma_N=\rho^2+\mathcal O(N^{-2})$, Eq.~\eqref{eq:01-error-uniform} gives
\[
\frac{\mu_2^{(0,1)}-\mathrm{Var}(\widehat r)}{\mathrm{Var}(\widehat r)-\mathrm{Var}(\bar r)}=\lambda+\mathcal O(N^{-1}),
\]
so that the error of $(0,1)$ relative to the contribution of missing data is set by $\lambda$, independently of $\alpha$ and $N$ to leading order. Under local dependence $\lambda=\mathcal O(N^{-1})$; under persistent dependence it tends to a positive limit (Corollary~\ref{prop:spatial-dependence}); in the special case of  equal pairwise correlation $c$, $\lambda=c/(1-c)+\mathcal O(N^{-1})$. Relative to the total variance, by contrast, the error of $(0,1)$ is at most of order $\rho^2$ whatever the spatial dependence, since $\mathrm{Var}(\widehat r)$ is at least $(1-(1-\alpha)^N)\mathrm{Var}(\bar r)$ and at least $\gamma_N\mathbf{E}s^2$, while $\gamma_N-\rho^2=\mathcal O(N^{-2})$. The sign of the error depends on $\lambda$. Since $\gamma_N-\rho^2=\rho^2(1+\alpha)/(\alpha N)+\mathcal O(N^{-3})$, to leading order $(0,1)$ overestimates the variance when $\lambda>(1+\alpha)/(\alpha N)$, as for correlated fields, so that it is then conservative, and underestimates it otherwise, as for independent sites, for which $\lambda\le1/(N-1)$. Since $\lambda$ depends only on the first and second moments of the field, it can be computed before an approximation is chosen.
\end{remark}

% ============================================================

\subsection{Summary of the truncations, and the role of higher truncations}
\label{sec:remarks}

Table~\ref{tab:truncation-properties} collects the properties of the truncations and relates them to the results of Section~\ref{sec:moments}, and Remark~\ref{rem:hierarchy} explains when the higher truncations can matter.

\begin{table}[!htb]
\centering
\small
\caption{Properties of the truncations under MCAR with uniform weights. Error orders are for the absolute error, with a bounded field and local dependence, in which the variance itself is of order $\rho^2$; under persistent dependence the error of $(0,1)$ is instead about $\rho^2\mathrm{Var}(\bar r)$ (Proposition~\ref{prop:uniform-decomp}). $^\dagger$Proved in Section~\ref{sec:moments}; the other orders follow from the term counting of Remarks~\ref{rem:translation} and~\ref{rem:hierarchy} and agree with exact calculations.}
\label{tab:truncation-properties}
\renewcommand{\arraystretch}{1.3}
\begin{tabular}{>{\raggedright\arraybackslash}p{0.11\linewidth} >{\raggedright\arraybackslash}p{0.20\linewidth} >{\centering\arraybackslash}p{0.09\linewidth} >{\centering\arraybackslash}p{0.12\linewidth} >{\raggedright\arraybackslash}p{0.32\linewidth}}
\hline
Truncation & Retains & Shift invariant & Error order & Intuition and link to Section~\ref{sec:moments} \\
\hline
$(0,0)$ & Denominator fixed at its mean & No & $\rho^2$\,$^\dagger$ & Omits the first-order term $-m_RV/m_S^2$, leaving the mean-induced term $\rho^2m_r^2$ \\
$(0,1)$ & Leading $\mathrm{Var}(S)$ and $\mathrm{Cov}(R,S)$ corrections & Yes & $\rho^4$\,$^\dagger$ & First shift-invariant truncation, and the estimator of \cite{Seshadri2018}; error $\approx\rho^2\mathrm{Var}(\bar r)$, or about $\lambda$ relative to the contribution of missing data \\
$(1,0)$ & Coupling term $D$, without $B$ & No & $\rho^2$ & Mixed and denominator families unbalanced; mean-dependent term of order $m_r^2N^{-1}$ \\
$(1,1)$ & $(0,1)$ plus the coupling term $D$ & No & $\rho^4$ & Adds terms of the same order as the $(0,1)$ error without completing them \\
$(0,2)$ & $(0,1)$ plus $\mathrm{Var}(V^2)$ and $\mathrm{Cov}(V,V^2)$ terms & No & $\rho^4$ & Its extra term $m_r^2\mathrm{Var}(V^2)/\alpha^4$ grows with the mean level \\
$(n,n+1)$, $n\ge1$ & All terms to total order $n+1$ & Yes & $\rho^4$, $\rho^6$, $\rho^6$, $\rho^8$, \dots & Converge for $\alpha>\tfrac12$ (Remark~\ref{rem:trunc-conv}); the order improves only every second step \\
Exact & Average over all reporting patterns & Yes & --- & Single sum over $K$; $(1-(1-\alpha)^N)\mathrm{Var}(\bar r)+\gamma_N\,\mathbf{E}s^2$ \\
\hline
\end{tabular}
\end{table}

\begin{remark}[When higher truncations matter]
\label{rem:hierarchy}
Under MCAR at fixed $\alpha$, with uniform weights and the local-dependence condition Eq.~\eqref{eq:local-dep-sec2}, the baseline truncations $\mu_2^{(0,0)}$ and $\mu_2^{(0,1)}$ are of order $N^{-1}$, whereas
\[
\mu_{1,2}=\mathcal O(N^{-2}),\qquad \mu_{2,1}=\mathcal O(N^{-2}),\qquad \kappa=\mathcal O(N^{-2}).
\]
The coupling and quadratic-denominator terms that distinguish $(1,0)$, $(1,1)$, and $(0,2)$ from the baseline are therefore one order smaller in $N$, but of the same order as the error of $(0,1)$ itself (Proposition~\ref{prop:rho4}). These truncations change the approximation at the order of its error without completing it, which is why they need not improve on $(0,1)$; a complete $\mathcal O(N^{-2})$ correction requires all terms to third order, as in $(2,3)$. The step from $(0,0)$ to $(0,1)$ is different: it completes the first-order terms, replacing the mean-induced term of $(0,0)$ by one in the spread of the site means, which vanishes when all site means are equal (Eq.~\eqref{eq:mar01-compact}).

Two regimes can make the higher corrections more important: small networks, where $N^{-1}$ is not small, and concentrated weights, where a few large $\beta_i$ inflate $W_2$ and the higher weight moments. The second is not characterized by $\rho$ alone. The error of $(0,1)$ involves the third moments $\mathbf{E}[U^2V]$, $\mathbf{E}[UV^2]$, and $\mathbf{E}[V^3]$ (Eq.~\eqref{eq:LA0-third-moments}), which are sums over the cubes of the weights, as in $W_3=\sum_i\beta_i^3$, rather than over their squares, so that its relative size depends on $(1-\alpha)W_3/(\alpha W_2)$ as well as on $\rho^2$. By the Cauchy--Schwarz inequality $W_3\ge W_2^2$, so this ratio is never smaller than $\rho^2$, and it equals $\rho^2$ only for uniform weights. This helps explain why the stress tests of Section~\ref{sec:numerics} show errors several times larger with concentrated than with uniform weights at comparable $\rho$. Long-range spatial dependence can also change the scaling of the mixed moments, so that the sharper bounds relying on Eq.~\eqref{eq:local-dep-sec2} need not hold, although Corollary~\ref{cor:mar-error} still applies. The rankings of the truncations found in Section~\ref{sec:numerics} are therefore properties of the cases examined, not general rules.
\end{remark}

\section{Numerical experiments}
\label{sec:numerics}

We use four numerical experiments to test the results of Sections~\ref{sec:moments}-\ref{sec:truncated}. The first, a homogeneous benchmark, tests the error rates of Proposition~\ref{prop:rho4}, the good-event and rare-event components of Theorem~\ref{thm:rho-error}, and, in Table~\ref{tab:convergence}, the convergence conditions of Proposition~\ref{prop:var-conv} and Remark~\ref{rem:trunc-conv}. The second compares the truncations with the exact variance for the IMD rainfall field, testing Proposition~\ref{prop:uniform-decomp}, the diagnostic $\lambda$ of Remark~\ref{rem:lambda}, and the ordering of the higher truncations discussed in Remark~\ref{rem:hierarchy}. The third adds a common shift to the field, testing the invariance criterion of Remark~\ref{rem:translation}. The fourth examines strong spatial correlation, small networks, and concentrated weights, the regimes identified in Proposition~\ref{prop:uniform-decomp} and Remark~\ref{rem:hierarchy}. Supplementary Information Figs.~S1-S3 give additional scaling results, verify the closed forms of the mixed moments, and decompose the low-order approximations term by term.

For the empirical experiments we use the daily $1^\circ\times1^\circ$ gridded rain-gauge dataset over Indian land regions produced by the India Meteorological Department (IMD) \cite{Rajeevan2006}. The 111 yearly arrays for 1901-2011 are concatenated across the years in sequence, giving a $T\times N$ record with $T=40515$ days and $N=357$ locations. Its site means and covariance matrix, pooled over all days of the year with normalization $1/T$, are treated as population moments in the MCAR calculations. Unless stated otherwise the weights are uniform, $\beta_i=1/N$, and the exact variance is then given by Eq.~\eqref{eq:exact-uniform-var} for any $N$. Only the small nonuniformly weighted networks of Fig.~\ref{fig:stress}(c) and Supplementary Information Fig.~S2 require complete enumeration of the reporting patterns.

\paragraph{Error rates and concentration of the denominator.}
The homogeneous benchmark uses a statistically homogeneous field, $m_i=m=1$, $\sigma_i^2=\sigma^2=1$, and $\sigma_{ij}=0$ for $i\ne j$, with uniform weights, which removes spatial and mean heterogeneity. In Fig.~\ref{fig:rho_master}(a), $N$ varies at fixed $\alpha=0.3$, so that $\rho=\sqrt{(1-\alpha)/(\alpha N)}$ decreases as the network grows. The exact errors $|\mathrm{Var}(\widehat r)-\mu_2^{(0,0)}|$ and $|\mathrm{Var}(\widehat r)-\mu_2^{(0,1)}|$ follow slopes $2$ and $4$, so that the theoretical rates $\rho^2$ and $\rho^4$ of Proposition~\ref{prop:rho4} are recovered. Relative errors behave differently. In this example,
\[
\mathrm{Var}(\widehat r)\sim \frac{\sigma^2}{\alpha N},
\qquad
\frac{|\mathrm{Var}(\widehat r)-\mu_2^{(0,0)}|}{\mathrm{Var}(\widehat r)}
\longrightarrow
(1-\alpha)\frac{m^2}{\sigma^2},
\]
so a vanishing absolute error of $(0,0)$ does not imply a vanishing relative error. Figure~\ref{fig:rho_master}(b) shows this limit for three availabilities, whereas the relative error of $(0,1)$ continues to decrease with $N$.

Figures~\ref{fig:rho_master}(c)-(d) show the two ingredients of Theorem~\ref{thm:rho-error}: as $N$ grows, the normalized denominator concentrates around zero, and the probability $\mathbf{P}(G^c)$ of the rare event falls rapidly, remaining below both the Chebyshev and Bernstein bounds. Supplementary Information Fig.~S1 varies $\alpha$ as well as $N$. There the normalized error $|\mathrm{Var}(\widehat r)-\mu_2^{(0,1)}|/\rho^4$ approaches $\sigma^2/(1-\alpha)$, independently of $m$. This follows from Eq.~\eqref{eq:01-error-uniform}: in the benchmark $\mathrm{Var}(\bar r)=\sigma^2/N$ and $\mathbf{E}s^2=\sigma^2(1-N^{-1})$, and since $\gamma_N=\rho^2\{1+(1+\alpha)/(\alpha N)+\mathcal O(N^{-2})\}$ \cite{Stephan1945}, the error is $\mu_2^{(0,1)}-\mathrm{Var}(\widehat r)=-\sigma^2\rho^4/(1-\alpha)+\mathcal O(\rho^6)$. The error is therefore of order $\rho^4$ at fixed $\alpha$, with a coefficient that depends on the availability, and since $\mathrm{Var}(\widehat r)\approx\sigma^2\rho^2/(1-\alpha)$, the relative error of $(0,1)$ is close to $-\rho^2$, an underestimate, as expected for independent sites (Remark~\ref{rem:lambda}). 

\paragraph{Convergence at fixed network size.}
Table~\ref{tab:convergence} uses the same benchmark with $N=10$. Along the shift-invariant truncations $(n,n+1)$, the truncated variances diverge for $\alpha<\tfrac12$, oscillate without converging at $\alpha=\tfrac12$, and converge for $\alpha>\tfrac12$. Along the $(0,n)$ family they drift away from the exact variance even for $\alpha>\tfrac12$, because their value on the all-missing pattern grows with $n$ (Remark~\ref{rem:trunc-conv}). Supplementary Information Table~S3 shows that, along the shift-invariant truncations, the order of the error in $N$ improves only at every second step, as the counting rule predicts (Remark~\ref{rem:translation}).

\begin{table}[t]
\centering
\small
\caption{Convergence at fixed $N=10$ for the benchmark with homogeneous field statistics ($m=\sigma=1$, uniform weights). (a)~The shift-invariant truncations $\mu_2^{(n,n+1)}$, compared with the exact variance: they diverge at $\alpha=0.45$, oscillate without converging at $\alpha=\tfrac12$, and converge at $\alpha=0.55$. (b)~At $\alpha=0.6$, the $(0,n)$ family drifts away from the exact variance because its value on the all-missing pattern grows with $n$, whereas the $(n,n+1)$ sequence converges. All values are computed as exact averages over the reporting process.}
\label{tab:convergence}
\vspace{4pt}
(a)\\[2pt]
\begin{tabular}{lcccccccc}
\hline
$\alpha$ & exact & $n=5$ & $n=10$ & $n=20$ & $n=40$ & $n=41$ & $n=80$ & $n=81$\\
\hline
$0.45$ & $0.2614$ & $0.2487$ & $0.2640$ & $0.4229$ & $477.4$ & $712.4$ & $4.47\times10^{9}$ & $6.68\times10^{9}$\\
$0.50$ & $0.2289$ & $0.2209$ & $0.2276$ & $0.2290$ & $0.2292$ & $0.2288$ & $0.2292$ & $0.2288$\\
$0.55$ & $0.2028$ & $0.1981$ & $0.2018$ & $0.2027$ & $0.2028$ & $0.2028$ & $0.2028$ & $0.2028$\\
\hline
\end{tabular}

\vspace{8pt}
(b)\\[2pt]
\begin{tabular}{lccccccc}
\hline
$\alpha=0.6$, exact $0.1818$ & $n=1$ & $n=2$ & $n=4$ & $n=8$ & $n=16$ & $n=32$ & $n=64$\\
\hline
$\mu_2^{(0,n)}$ & $0.1667$ & $0.1747$ & $0.1896$ & $0.2112$ & $0.2409$ & $0.3210$ & $0.6360$\\
$\mu_2^{(n,n+1)}$ & $0.1552$ & $0.1751$ & $0.1789$ & $0.1809$ & $0.1816$ & $0.1817$ & $0.1818$\\
\hline
\end{tabular}
\end{table}

\begin{figure}
    \centering
    \makebox[\textwidth][c]{%
        \includegraphics[width=1.12\textwidth,height=0.72\textheight,keepaspectratio]{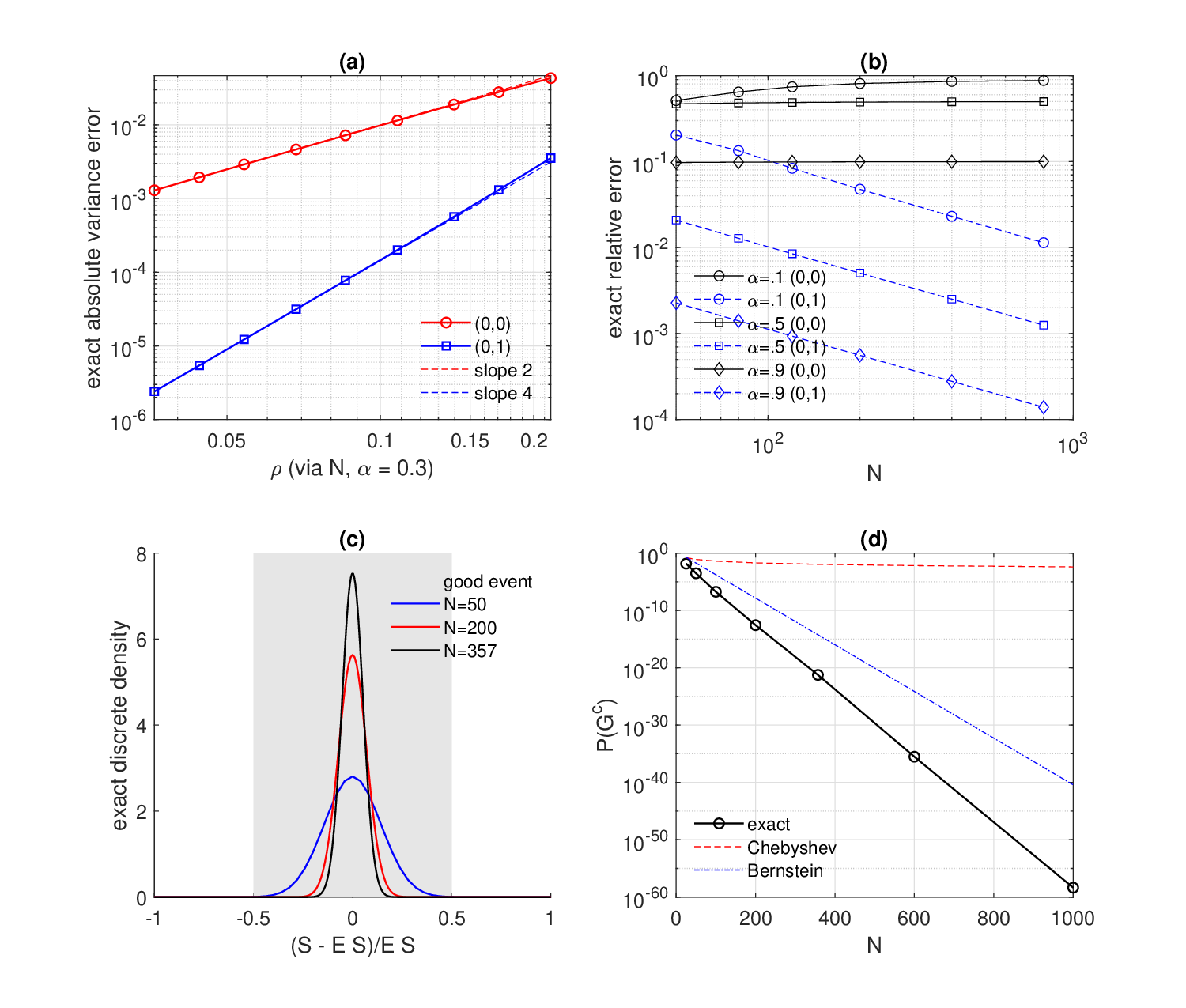}%
    }
    \caption{Scaling of the truncation error relative to exact calculation and the denominator concentration for the homogeneous benchmark (i.i.d.\ sites, $m=\sigma=1$, uniform weights). (a) Errors of $(0,0)$ and $(0,1)$ relative to the exact variance against $\rho$ at fixed $\alpha=0.3$, for $N=50,80,120,200,320,500,800,1200,1800$, compared with reference slopes $2$ and $4$. (b) Corresponding relative errors for $N=50,80,120,200,400,800$ at $\alpha=0.1,0.5,0.9$: the relative error of $(0,0)$ approaches the nonzero limit $(1-\alpha)m^2/\sigma^2$, whereas that of $(0,1)$ continues to decrease. (c)~Exact distribution of the normalized denominator $(S-\mathbf{E}S)/\mathbf{E}S$, which is binomial for uniform weights, for $N=50,200,357$ at $\alpha=0.5$; the shaded interval is the good event $|S-\mathbf{E}S|\le\tau\mathbf{E}S$ with $\tau=0.5$. (d)~Exact $\mathbf{P}(G^c)$ for $N=25,50,100,200,357,600,1000$, with the Chebyshev and Bernstein bounds for the same $\alpha$ and $\tau$.}
    \label{fig:rho_master}
\end{figure}

\paragraph{Validation using the IMD rainfall field.}
The second experiment uses the full spatial heterogeneity and covariance of the IMD field. Figure~\ref{fig:imd_validation} compares the exact variance with $(0,0)$, $(0,1)$, $(1,0)$, and $(1,1)$ for $0.1\le\alpha\le1$. All four approximations lie within a few percent of the exact variance, and $(0,1)$ is the most accurate at every availability. The error is not monotonic in the truncation order: at $\alpha=0.1$ it is $1.6\%$ for $(0,1)$, but $3.2$-$4.8\%$ for the other three, so the additional coupling terms do not improve on $(0,1)$ (Remark~\ref{rem:hierarchy}).

The term-by-term decomposition in Supplementary Information Fig.~S3 shows why. The off-diagonal covariance contribution dominates the variance of this field, which lies far from the local-dependence regime of Eq.~\eqref{eq:local-dep-sec2}, with $\Sigma_{\mathrm{off}}\approx15N^2$ rather than $\mathcal O(N)$. The covariance part of $\mu_{2,1}$ is therefore not of higher order, and at $\alpha=0.1$ the coupling terms added by $(1,0)$ and $(1,1)$, $-2\mu_{2,1}/\alpha^3\approx-1.7$ and $\kappa/\alpha^4\approx0.5$ (mm day$^{-1}$)$^2$, are each larger than the error of $(0,1)$ itself, $0.29$ (mm day$^{-1}$)$^2$.

Relative to the contribution of missing data the errors are much larger (Fig.~\ref{fig:imd_validation}(c)). At $\alpha=0.1$, for example, the contribution of missing data is $3.02$ out of an exact variance of $18.24$ (mm day$^{-1}$)$^2$, and the error of $(0,1)$ is about $10\%$ of that contribution, although only $1.6\%$ of the total variance. The distinction matters because the contribution of missing data is the part of the uncertainty due to incomplete coverage, and an approximation that is accurate for the total variance can still misstate that part appreciably. The error of $(0,1)$ is dominated by the first term of Eq.~\eqref{eq:01-error-uniform}, $\rho^2\mathrm{Var}(\bar r)$. It arises because $(0,1)$ multiplies $\rho^2$ by $\mathbf{E}s^2+\mathrm{Var}(\bar r)$ rather than by the spatial variance $\mathbf{E}s^2$ alone, and so attributes part of the variability of the spatial average itself to missing data. Across the range of availability, the four approximations are in error by about $10$-$30\%$ of the contribution of missing data, with $(0,1)$ the smallest at about $10$-$12\%$. This last range agrees with the diagnostic of Remark~\ref{rem:lambda}, $\lambda\approx0.13$ for this field. As expected for a correlated field, $(0,1)$ overestimates the variance, so its estimate of the uncertainty due to missing data is conservative.

\begin{figure}
    \centering
    \makebox[\textwidth][c]{%
        \includegraphics[width=1.12\textwidth,height=0.72\textheight,keepaspectratio]{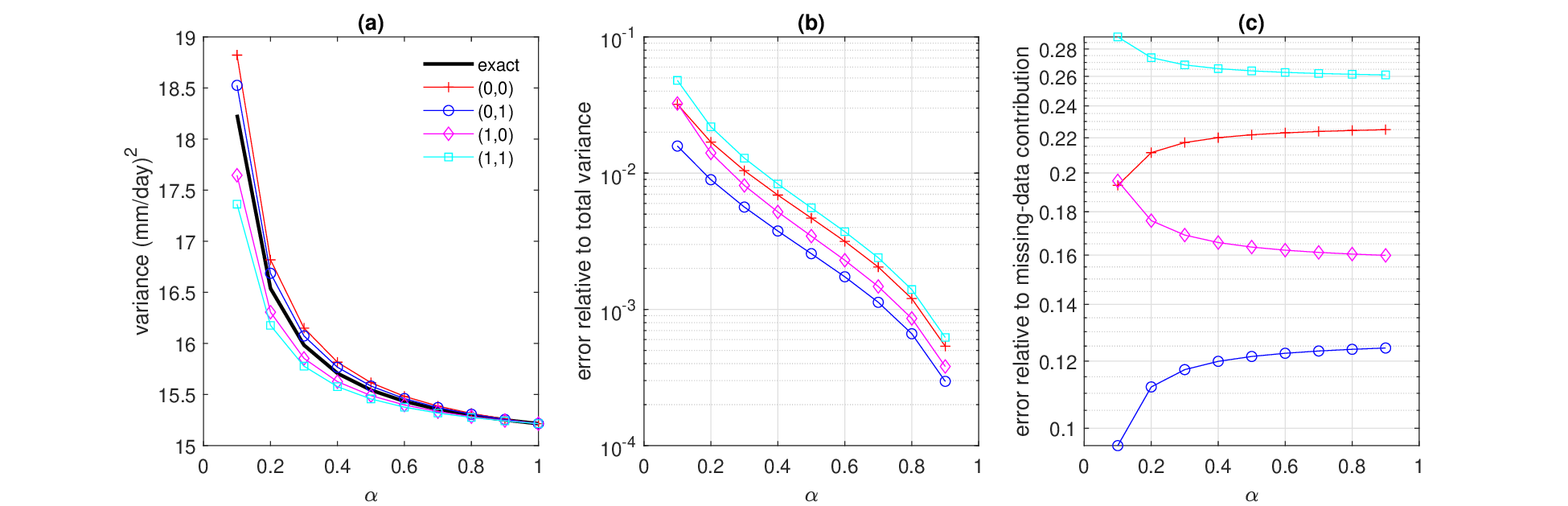}%
    }
    \caption{Low-order truncations for the IMD rainfall field, using population moments from the 1901-2011 daily record ($T=40515$, $N=357$) with uniform weights, for $\alpha=0.1,0.2,\ldots,1$. The exact MCAR variance is computed from Eq.~\eqref{eq:exact-uniform-var}. (a)~Exact variance and the $(0,0)$, $(0,1)$, $(1,0)$, and $(1,1)$ approximations. (b)~Their absolute errors relative to the total variance, for $\alpha<1$; at $\alpha=1$ all approximations are exact. (c)~The same errors relative to the contribution of missing data, $\mathrm{Var}(\widehat r)-\mathrm{Var}(\bar r)$. The errors are small relative to the total variance, which is dominated by the variance of the average with complete reporting, but not relative to the contribution of missing data; $(0,1)$ is the most accurate of the four throughout.}
    \label{fig:imd_validation}
\end{figure}

\paragraph{Dependence on the mean level of the field.}
For the rainfall field the difference between $(0,0)$ and $(0,1)$ is modest, because its mean level is small compared with its variability, $m_r\approx0.3\,\sigma_{\rm RMS}$. To examine the dependence on the mean level directly, without changing the covariance structure or the spatial heterogeneity of the means, we add the same constant $b$ at every site and time, $r_i^{(b)}(t)=r_i(t)+b$, and shift the all-missing convention correspondingly, $\widehat r^{(b)}=\widehat r+b$. The exact variance is then unchanged, and so are the covariance matrix and $\mathrm{Var}_m$, and hence $\mu_2^{(0,1)}$. By contrast, $(0,0)$ retains the mean-induced reporting contribution; under uniform weights,
\begin{equation}
\mu_2^{(0,0)}-\mu_2^{(0,1)}
=
\frac{1-\alpha}{\alpha}\frac{m_r^2}{N}.
\label{eq:translation-gap}
\end{equation}
Figure~\ref{fig:translation} confirms that only $(0,1)$ is unaffected by the shift, as Remark~\ref{rem:translation} predicts. The error of $(0,0)$ changes most, since its mean-dependent term is of order $N^{-1}$, whereas those of $(1,1)$ and $(0,2)$ are of order $N^{-2}$. Invariance concerns the lack of dependence on the mean, and is an exact notion that does not depend on the size of the error: over this range the error of $(1,1)$, which is negative for the rainfall field, shrinks in magnitude as the mean increases.

\begin{figure}
    \centering
    \makebox[\textwidth][c]{%
        \includegraphics[width=1.12\textwidth,height=0.72\textheight,keepaspectratio]{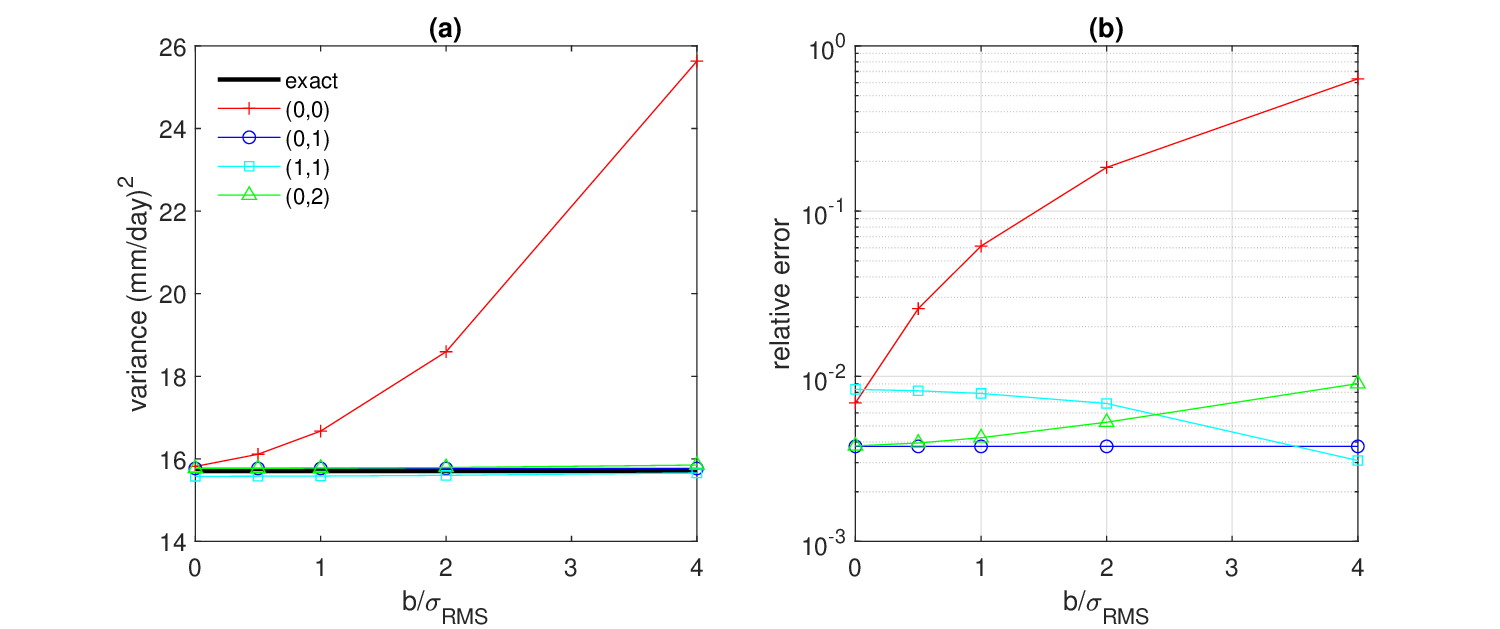}%
    }
    \caption{Common-shift experiment for the IMD field (with $\alpha=0.4$ and uniform weights). The same constant $b$ is added at every site; the horizontal coordinate is $b/\sigma_{\rm RMS}=0,0.5,1,2,4$, where $\sigma_{\rm RMS}=(N^{-1}\sum_i\sigma_i^2)^{1/2}$ for the original field. (a)~The exact variance and the $(0,0)$, $(0,1)$, $(1,1)$, and $(0,2)$ approximations. (b)~Their absolute relative errors; only that of $(0,1)$ is unchanged by the shift, like the exact variance.}
    \label{fig:translation}
\end{figure}

\paragraph{Stress tests.}
The low-order approximations become less accurate in three distinct ways, each with a different cause, and Fig.~\ref{fig:stress} illustrates each in turn. Together the tests show how small $\rho$ must be for the low-order approximations to be reliable, and that $\rho$ alone is not a sufficient guide when the weights are concentrated.

First, strong spatial correlation affects the first-order approximation even when the denominator is well concentrated. In Fig.~\ref{fig:stress}(a), all pairs of $N=100$ sites have the same correlation $c$. Relative to the total variance, the error of $(0,1)$ stays below $\rho^2=0.015$ for every $c$, but relative to the contribution of missing data it rises from $8\%$ at $c=0.1$ to nearly nine times that contribution at $c=0.9$. It closely follows the leading term $\rho^2\mathrm{Var}(\bar r)$ of Proposition~\ref{prop:uniform-decomp}, and hence the diagnostic $\lambda$ of Remark~\ref{rem:lambda}, which here is close to $c/(1-c)$. Strong correlation raises the variance of the average with complete reporting, as in the persistent case of Corollary~\ref{prop:spatial-dependence}, but reduces the spatial variance of each realization, and hence the contribution of missing data, which $(0,1)$ does not capture.

Second, Fig.~\ref{fig:stress}(b) reduces the network size under uniform weights, with site means between $0.6$ and $1.4$ and unit variances. The relative error of $(0,1)$ grows roughly as $\rho^2$, from below $1\%$ at $N=200$ ($\rho=0.09$) to about $18\%$ at $N=10$ ($\rho=0.39$); in the smallest networks, $N\le6$, it falls well below $\rho^2$. The additional denominator term retained by $(0,2)$ reduces the error by more than an order of magnitude for $N\ge30$, and $(0,2)$ remains the most accurate of the three truncations down to $N=6$; the dips in its curve mark the two changes of sign of its error. This advantage depends on the mean level, since $(0,2)$ is not invariant to a common shift (Remark~\ref{rem:translation}): for larger $m_r/\sigma$ its extra $m_r^2$ term would dominate. The panel therefore shows how a higher denominator correction can lower the truncation error in a particular case, but it gives no general rule for small networks.

Third, Fig.~\ref{fig:stress}(c) isolates the effect of weight concentration in a homogeneous independent network of $N=20$ sites. Giving site~1 less than the uniform weight $1/20$ changes $\rho$ little, and the errors stay close to their uniform-weight values. As site~1 receives more weight, $W_2$ and $\rho$ increase and the errors eventually become large: that of $(0,1)$ exceeds $100\%$ when site~1 carries half the weight ($\rho=0.63$). No single truncation is best throughout: $(0,2)$ is the most accurate when site~1 carries at most a tenth of the weight, and $(1,1)$ when it carries $30\%$ or more. Because panels (b) and (c) share their axes, the two can be compared directly. At $\rho\approx0.5$ the error of $(0,1)$ is several times larger with concentrated weights than with uniform weights, $75\%$ at $\rho=0.52$ against $14\%$ at $\rho=0.50$, so that higher weight moments, and not $\rho$ alone, then govern the error, as noted in Remark~\ref{rem:hierarchy}.

\begin{figure}
    \centering
    \makebox[\textwidth][c]{%
        \includegraphics[width=1.12\textwidth,height=0.72\textheight,keepaspectratio]{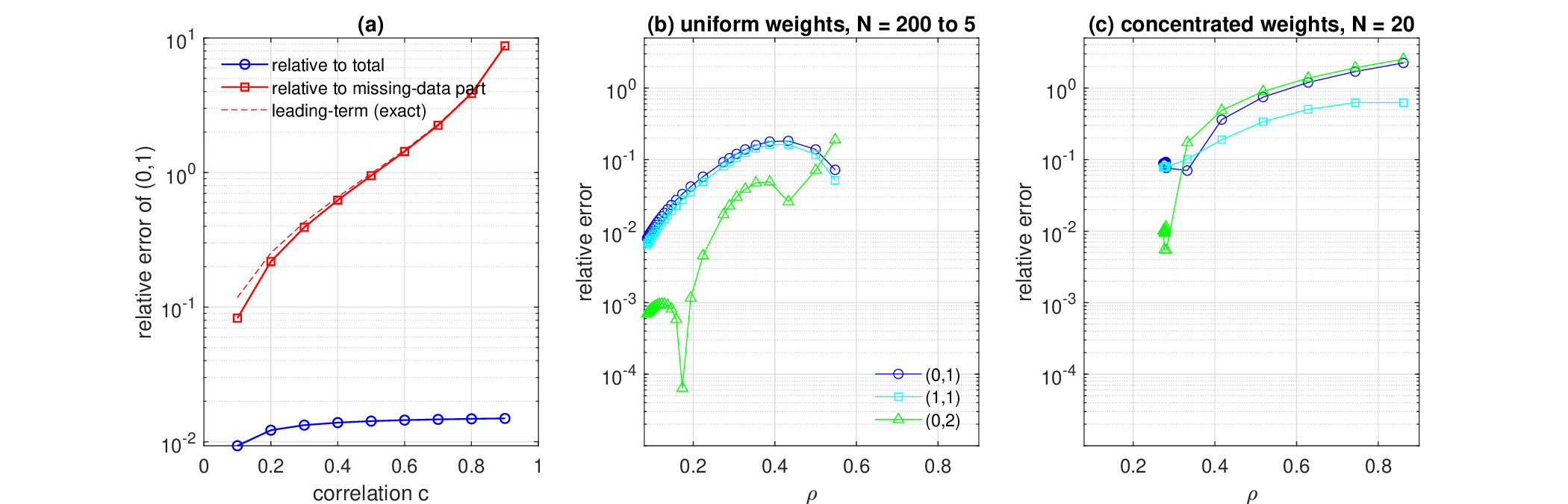}%
    }
    \caption{Three regimes that stress the low-order approximations, at fixed $\alpha=0.4$. (a)~Strong spatial correlation: uniform weights, $N=100$ sites with unit means and variances, and equal pairwise correlation $c=0.1,\ldots,0.9$. The error of $(0,1)$ is shown relative to the total variance and to the contribution of missing data, $\mathrm{Var}(\widehat r)-\mathrm{Var}(\bar r)$, together with the leading-term prediction $\rho^2\mathrm{Var}(\bar r)$ of Proposition~\ref{prop:uniform-decomp} relative to the same contribution. (b)~Small networks: uniform weights, $\Sigma=I$, and site means spanning $0.6$-$1.4$ linearly, for $N=200,190,\ldots,20$ and $N=18,16,\ldots,6,5$ ($\rho=0.09$-$0.55$). (c)~Weight concentration: a homogeneous independent network with $N=20$, unit means and variances, and $\beta_1$ between $0.004$ and $0.70$, below and above the uniform value $1/20$, with the remaining weight divided equally among the other sites ($\rho=0.27$-$0.86$); the exact variance is obtained by enumerating all $2^{20}$ reporting patterns. Panels (b) and (c) show absolute relative errors against $\rho$ on common axes.}
    \label{fig:stress}
\end{figure}

\section{Discussion and conclusions}
\label{sec:discussion}

When observations of rainfall, temperature, or other climate variables are missing randomly, their spatial average is a ratio of random variables, and in earlier work its variance was estimated by truncating the Taylor series of the ratio and taking expectations \cite{Seshadri2018}. The resulting estimator compared well with simulations, except when the probability of missing data was large. Here we have examined why, and more generally when such approximations can be trusted for spatial averages with missing data, how their accuracy depends on the reporting process, the weights, and the spatial correlation of the field, and what this implies for self-normalized estimators. Three questions must be distinguished, and they have different answers: whether the series converges as more terms are retained, how large the error of a given truncation is, and how this error declines as the network grows. With observations missing completely at random, the series converges for every set of reported observations only if the availability $\alpha$ exceeds one half, and the truncated variances then approach the exact variance only if the two kinds of correction are increased together (Section~\ref{sec:conv}). Convergence is not, however, needed for accuracy. The error of the truncations that correct only for the denominator is controlled by $\rho$, the coefficient of variation of the total weight of reported observations, and for uniform weights and a bounded field it vanishes as the network grows, at any availability and for any spatial dependence (Section~\ref{sec:accuracy}). When correlations decay with distance, the error of the first-order approximation, i.e. the variance estimator of \cite{Seshadri2018}, is of order $\rho^4$, compared with $\rho^2$ for the simplest approximation, which helps explain the good performance of the first-order approximation. Since $\rho^2=\tfrac{1-\alpha}{\alpha}\sum_i\beta_i^2$ grows as availability falls and as weight is concentrated, as in the optimal weights used there, this also explains why it departed from simulations at small availability.

The sequence of truncations also shows which higher corrections should be retained. Because the weights of the reported observations always sum to one, adding a constant to the field adds the same constant to the spatial average and leaves its variance unchanged. Only truncations that treat the numerator and denominator to the same order share this property, beginning with the first-order approximation. Other truncations to the variance, including the simplest approximation, contain terms that grow with the mean of the field and are absent from the exact variance (Section~\ref{sec:shift}). Higher truncations need not be more accurate, however, since they can retain some corrections of a given order while omitting others of the same order, as when $(1,1)$ adds the leading coupling term but not the matching denominator term. Even along the shift-invariant truncations $(0,1)$, $(1,2)$, $(2,3)$, and $(3,4)$, the order of the error improves only at every second step, being $\rho^4$, $\rho^4$, $\rho^6$, and $\rho^6$, respectively, under local dependence (Remark~\ref{rem:translation} and Table~\ref{tab:truncation-properties}).

Spatial covariance affects the variance in a way that $\rho$ does not measure. For uniform weights, the variance separates exactly into the variance of the average with complete reporting, $\mathrm{Var}(\bar r)$, and a contribution from missing data (Proposition~\ref{prop:uniform-decomp}). Strong spatial correlation increases the first but decreases the second, because each realization of the field then varies little from site to site, so that losing a site matters less. The first-order approximation does not capture this decrease. For correlated fields its error is approximately $\rho^2\mathrm{Var}(\bar r)$, whereas the contribution of missing data is approximately $\rho^2\mathbf{E}s^2$, where $\mathbf{E}s^2$ is the expected variance of the field across sites within a single realization. Relative to the contribution of missing data, the error is therefore about $\lambda=\mathrm{Var}(\bar r)/\mathbf{E}s^2$, up to terms of order $N^{-1}$. This ratio is small when the sites are nearly independent, but it can exceed one when they are strongly correlated. Relative to the total variance, by contrast, the error is at most of order $\rho^2$, whatever the spatial dependence. Because the approximation overestimates the variance of correlated fields, first-order estimates of the uncertainty due to missing data are conservative for such fields; for rainfall over India, $\lambda\approx0.13$. Although correlations in the rainfall field do not decay quickly enough with distance for the $\rho^4$ rate to apply, the first-order approximation was still the most accurate of the four we computed for it. When spatial dependence is persistent, so that correlations do not weaken with distance, the contribution of missing data vanishes as the network grows while the variance does not, and uncertainty in the spatial average then arises mainly from variability of the field rather than from incomplete coverage.

These results suggest a simple procedure, summarized in Table~\ref{tab:practical}. With uniform weights, the exact variance can be computed directly, as a single sum over the number of reported observations (Eq.~\eqref{eq:exact-uniform-var}), from a few aggregates accumulated from the data record. With general weights, the first-order approximation is the natural choice. It is unaffected by the mean level of the field, and it requires only the means and covariances of the observations, with independent measurement noise entering only through the site variances \cite{Seshadri2018}. For uniform weights its error relative to the total variance is of order $\rho^2$, so a target relative error $\varepsilon$ requires, roughly, $\rho^2\le\varepsilon$, that is, at least $(1-\alpha)/(\alpha\varepsilon)$ sites, or about $100$ sites for $\alpha=0.5$ and $\varepsilon=1\%$. For unequal weights the same rule, applied to the effective number of sites $1/\sum_i\beta_i^2$, is only indicative, since the error then also depends on $\sum_i\beta_i^3$ (Remark~\ref{rem:hierarchy}). When $\rho$ is not small, an exact calculation by enumeration, or simulation, is preferable, and the higher truncations are better treated as diagnostics than as improvements. Since $\rho$ fixes only the order of the error, the first-order approximation should also be checked in this way when the weights are concentrated, even at moderate $\rho$; in our examples its relative error exceeded $100\%$ when one site carried half the weight. In designing networks, increasing availability or spreading the weights more evenly reduces $\rho$, whereas optimal weights, which emphasize a small fraction of locations \cite{Seshadri2018}, increase it. Adding highly correlated observations stabilizes the denominator but adds little independent information. Unlike Monte Carlo simulation of missing data, which carries sampling error and must be repeated for each configuration, these formulas are deterministic once the statistics of the field are known.

\begin{table}[t]
\centering
\small
\caption{Practical guidance for the variance of a weighted spatial average under MCAR. Here $\rho^2=\tfrac{1-\alpha}{\alpha}\sum_i\beta_i^2$; the thresholds are indicative, taken from the examples of Section~\ref{sec:numerics}.}
\label{tab:practical}
\renewcommand{\arraystretch}{1.3}
\begin{tabular}{>{\raggedright\arraybackslash}p{0.27\linewidth} >{\raggedright\arraybackslash}p{0.38\linewidth} >{\raggedright\arraybackslash}p{0.25\linewidth}}
\hline
Situation & Recommended calculation & Basis \\
\hline
Uniform weights, any spatial covariance & Exact single-sum expression, Eq.~\eqref{eq:exact-uniform-var}; no simulation needed & Section~\ref{sec:exact-uniform} \\
General weights, $\rho$ small (below about $0.2$) and diffuse weights & First-order approximation $(0,1)$; for uniform weights the relative error is of order $\rho^2$, so a target relative error $\varepsilon$ requires roughly $\rho^2\le\varepsilon$, a rule that is indicative for unequal weights & Proposition~\ref{prop:rho4}, Remarks~\ref{rem:lambda} and~\ref{rem:hierarchy}, Fig.~\ref{fig:rho_master} \\
Correlated field (e.g.\ rainfall) & Compute $\lambda=\mathrm{Var}(\bar r)/\mathbf{E}s^2$; $(0,1)$ is adequate and conservative for the total variance, but its error relative to the contribution of missing data is about $\lambda$ & Remark~\ref{rem:lambda}, Figs.~\ref{fig:imd_validation}(c) and~\ref{fig:stress}(a) \\
Large mean level relative to the variability of the field & Use $(0,1)$ or an exact calculation; avoid $(0,0)$, $(1,1)$, and $(0,2)$, which are not invariant to a common shift & Section~\ref{sec:shift}, Fig.~\ref{fig:translation} \\
$\rho$ not small (small network or low availability) & Exact calculation by enumeration or simulation; treat $(1,0)$, $(1,1)$, and $(0,2)$ as diagnostics & Remark~\ref{rem:hierarchy}, Fig.~\ref{fig:stress}(b) \\
Concentrated weights, e.g.\ optimized weights & Check $(0,1)$ against enumeration or simulation even at moderate $\rho$ & Remark~\ref{rem:hierarchy}, Fig.~\ref{fig:stress}(c) \\
\hline
\end{tabular}
\end{table}

Several of these results apply to self-normalized estimators more generally, including H\'ajek estimators in survey sampling and self-normalized importance sampling, since they use only that the spatial average is a weighted mean of the reported values with weights summing to one. Missing data with a common availability is Bernoulli sampling, and the spatial average is then a H\'ajek estimator \cite{Sarndal1992}. For such estimators the delta method is usually justified by a convergent Taylor series, but the series diverges with positive probability whenever the denominator can exceed twice its mean, as under Bernoulli sampling with a sampling fraction of at most one half, or in importance sampling with unbounded weights. Our error bound nevertheless holds at every sample size, which complements the justification given for the ratio estimator under simple random sampling \cite{DavidSukhatme1974}. The invariance criterion also applies generally, so that corrections that add only denominator terms, as is common, introduce a spurious dependence on the mean of the observed values.

The results also show what can be learned from the weights alone. In survey practice the coefficient of variation of the denominator, which is $\rho$ here, is used informally to judge whether a linearized variance can be trusted \cite{Cochran1977}, and $1/\sum_i\beta_i^2$ is Kish's effective sample size \cite{Kish1965}, whose use in importance sampling has been questioned \cite{Elvira2022}. Such measures are partly justified, since $\rho$ controls the truncation error. The error relative to the contribution of missing data is, however, set by $\lambda$, which depends on the observed values and cannot be inferred from the weights, and with concentrated weights $\rho$ alone does not determine the accuracy. When the field is treated as fixed, so that only the sampling varies, $\lambda$ vanishes, and for uniform weights the first-order approximation then underestimates the variance by a relative amount $(1+\alpha)/(\alpha N)$ to leading order. Some of the small-sample underestimation long noted for linearized variances of ratio estimators \cite{Cochran1977} is therefore intrinsic to the linearization. The explicit formulas, the threshold $\alpha>\tfrac12$, the error rates, and $\lambda$, by contrast, rely on the model of missing data used here, and $\lambda$ also on uniform weights.

Several limitations should be noted. The expressions developed here describe prior estimates of variance, taken over both the field and the missing data, which are useful when it is not known which observations will be available \cite{Seshadri2018}. Of course, knowing which observations are actually available at a given time would affect posterior estimates of variance. Separating these two sources of variability for general weights, as we have done for uniform weights, for example by expanding about the average with complete reporting, is a natural next step. We have not treated the bias that missing data introduce, which arises with nonuniform weights when the weights are correlated with the site means \cite{Seshadri2018}. The sharper error rates are proved only for uniform weights, and tests with nonuniform weights used a small network, so realistic area-weighted networks remain to be examined. In applications the statistics of the field, treated here as known, must be estimated, which adds uncertainty distinct from truncation error; short records may require regularized covariance estimates. Pooling these statistics over the year also counts the seasonal cycle as variability of the field, which seasonally stratified statistics would avoid. Finally, sites were assumed to report independently of one another. Persistent outages, clustered gaps, or reporting that depends on the field change the statistics of the denominator and may prevent $\rho$ from vanishing. Reporting probabilities that differ between sites are covered by Remark~\ref{prop:het-alpha}, but other forms of missing data, and availabilities or weights that change with network size, require further analysis. A companion paper will consider large-sample and non-asymptotic distribution theory.

In summary, low-order formulas approximate the variance of a spatial average well when the total weight of reported observations varies little from one period to the next, as when availability is high and the weights are spread over many sites, even when the underlying series diverges. The first-order approximation of \cite{Seshadri2018} is then the natural choice. It is unaffected by the mean of the field, higher truncations need not improve on it, and for uniform weights and correlations that decay with distance its error decreases rapidly as the network grows. For correlated fields such as rainfall over India, it overestimates the contribution of missing data, while remaining accurate, and conservative, for the variance as a whole. For uniform weights the exact variance is a single sum and needs no approximation, and where the total reported weight is variable, or the weights are concentrated, the first-order approximation should be checked by enumeration or simulation.

\paragraph{Code and data availability.}
The $1^{\circ}\times1^{\circ}$ gridded daily rainfall over India used in this study was produced by the India Meteorological Department \cite{Rajeevan2006}. The public MATLAB repository contains the code and data required to reproduce the four main-text and three Supplementary Information figures and the values in Table~\ref{tab:convergence} and Supplementary Information Table~S3, with one script for each figure and a small set of helper functions, including the exact reporter-count calculation used for the uniform-weight benchmarks. The code is archived on Zenodo \cite{SeshadriCode2026} and developed on \href{https://github.com/akseshadri/convergent-moment-hierarchy}{GitHub}.

\appendix
\allowdisplaybreaks

\section{Tools and a decomposition used in the proofs}
\label{app:tools}

The appendices show why the statements of Section~\ref{sec:moments} hold. Readers mainly interested in applying the approximations need refer only to the statements themselves and Tables~\ref{tab:three-results}, \ref{tab:truncation-properties}, and~\ref{tab:practical}. The proofs apply a few standard results, which we describe here in words.

\paragraph{Tools.}
The root-mean-square size of a random variable $X$ is $\lVert X\rVert_2=(\mathbf{E}X^2)^{1/2}$, and more generally $\lVert X\rVert_p=(\mathbf{E}|X|^p)^{1/p}$; a larger $p$ gives more weight to rare large values. The Cauchy-Schwarz inequality, $|\mathbf{E}[XY]|\le\lVert X\rVert_2\lVert Y\rVert_2$, bounds the average of a product by the sizes of its factors. Chebyshev's inequality, $\mathbf{P}(|X|>t)\le\mathbf{E}X^2/t^2$, says that a variable rarely strays far beyond its root-mean-square size. Bernstein's inequality sharpens this for a sum of many small independent terms, such as the fluctuation $V=S-m_S$ of the total reported weight: if $V=\sum_iX_i$ with independent $X_i$ of mean zero, $|X_i|\le M$, and $\sigma_V^2=\sum_i\mathrm{Var}(X_i)$, then
\begin{equation}
\mathbf{P}(|V|>t)
\le
2\exp\left(
-\frac{t^2/2}{\sigma_V^2+Mt/3}
\right),
\label{eq:bernstein}
\end{equation}
so that large departures become exponentially unlikely. Under MCAR, $X_i=\beta_i(s_i-\alpha)$, so we may take $M=\max_i\beta_i$ and $\sigma_V^2=\alpha(1-\alpha)W_2=\mathrm{Var}(S)$. Dominated convergence allows an infinite series to be averaged term by term when its partial sums are bounded by a single variable of finite mean. Finally, $\mathcal O(N^{-k})$ denotes a quantity bounded by a constant times $N^{-k}$, with the constant independent of $N$. We also use the identity
\begin{equation}
\mathrm{Var}(X)-\mathrm{Var}(Y)
=
\mathrm{Cov}(X-Y,X+Y)
=
2\,\mathrm{Cov}(X-Y,X)-\mathrm{Var}(X-Y),
\label{eq:var-diff}
\end{equation}
which follows from bilinearity of covariance, to compare an exact variance with that of a truncation.

\paragraph{Uniform weights: a decomposition used in Appendices~\ref{app:proof-cor-mar} and~\ref{app:proof-rho4}.}
With uniform weights, $S=N^{-1}\sum_is_i$, $m_S=\alpha$, $\mathrm{Var}(S)=\alpha(1-\alpha)/N$, and $\rho^2=(1-\alpha)/(\alpha N)$. The moment bound used in Lemma~\ref{lem:term-order} gives $\lVert V\rVert_{2k}=\mathcal O(N^{-1/2})$ for every fixed integer $k\ge1$, and hence
\begin{equation}
\mathbf{E}|V|^{2k}
=
\mathcal O(N^{-k})
=
\mathcal O(\rho^{2k}),
\label{eq:V-even-fixedalpha}
\end{equation}
since $N^{-1}$ and $\rho^2$ are proportional at fixed $\alpha$. For a bounded field, $|r_i|\le\Lambda$, we have $|R|\le\Lambda S\le\Lambda$ and $|m_R|\le\Lambda\,\mathbf{E}S=\Lambda\alpha$, so that $|U|\le2\Lambda$, while $|V|\le1$ because $S$ and $m_S$ lie between $0$ and $1$.

The numerator fluctuation separates into a contribution due to the field and a contribution due to reporting. Writing $x_i:=s_i-\alpha$ and $s_i=\alpha+x_i$,
\[
R=\frac{\alpha}{N}\sum_i r_i+\frac1N\sum_i x_ir_i=m_R+\alpha\xi+\eta,
\qquad
\xi:=\frac1N\sum_i(r_i-m_i),
\qquad
\eta:=\frac1N\sum_ix_ir_i,
\]
so that $U=\alpha\xi+\eta$ and $V=N^{-1}\sum_ix_i$. Since reporting is independent of the field, $\xi$, which depends only on the field, is independent of $V$, which depends only on reporting. Moreover, given the field, $\eta$ is an average of independent, centered terms $x_ir_i$, so that $\mathbf{E}(\eta\mid r_1,\ldots,r_N)=0$ and hence $\mathbf{E}[\xi\eta]=0$. The variance of $\xi$ is
\[
\mathbf{E}\xi^2=\frac1{N^2}\Big(\sum_i\sigma_i^2+\sum_{i\ne j}\sigma_{ij}\Big)\le\frac1{N^2}\Big(\sum_i\sigma_i^2+\sum_{i\ne j}|\sigma_{ij}|\Big),
\]
which is $\mathcal O(N^{-1})$ under the local-dependence condition, Eq.~\eqref{eq:local-dep-sec2}, since $\sum_i\sigma_i^2=\mathcal O(N)$ for a bounded field. Thus the spatial dependence of the field enters only through $\xi$, and the randomness of reporting only through $\eta$ and $V$.

\section{When the series converges (Proposition~\ref{prop:var-conv}), and how large its terms are (Lemma~\ref{lem:term-order})}
\label{app:proofs-sec2}

Proposition~\ref{prop:var-conv} states that the series for the ratio converges, and may be averaged term by term, when the total reported weight never deviates more than a fixed fraction from its mean. The idea is that the series is then dominated by a convergent geometric series. Lemma~\ref{lem:term-order} states how large the successive terms are on average, which is what the error bounds of Section~\ref{sec:accuracy} need.

\paragraph{Convergence.}
We write $f(R,S)=\sum_{n\ge0}a_n+\sum_{n\ge0}b_n$, with
\[
a_n=(-1)^n\frac{(R-\mathbf{E}R)(S-\mathbf{E}S)^n}{(\mathbf{E}S)^{n+1}},
\qquad
b_n=(-1)^n\frac{\mathbf{E}R\,(S-\mathbf{E}S)^n}{(\mathbf{E}S)^{n+1}}.
\]
On every nonempty reporting pattern, condition~\eqref{eq:suff-a.s} bounds the absolute terms by those of a geometric series with ratio $\bar q<1$, so that the series converges absolutely to $f$, and its Cauchy product converges absolutely to $f^2$ \cite{Rudin1976}. Summing the geometric bounds, the partial sums are dominated by a single variable,
\[
\sum_{n}\big(|a_n|+|b_n|\big)
\le\frac{|R-\mathbf{E}R|+|\mathbf{E}R|}{\mathbf{E}S\,(1-\bar q)},
\]
which has a finite second moment because $\mathbf{E}r_i^2<\infty$. Dominated convergence therefore allows the second-moment series to be averaged term by term over the nonempty patterns. The all-missing pattern, on which the estimator takes the conventional value $r_*$, is treated separately, and the characterization under MCAR is shown in the text following the proposition.

\paragraph{Size of the terms.}
Centering a variable does not increase its root-mean-square size, so the $n$th denominator term satisfies
\[
\left\|
\frac{m_R(V^n-\mathbf{E}V^n)}{m_S^{n+1}}
\right\|_2
\le
\frac{|m_R|\,\lVert V^n\rVert_2}{m_S^{n+1}}
=
\frac{|m_R|}{m_S}
\left(\mathbf{E}|V/m_S|^{2n}\right)^{1/2}
=
\frac{|m_R|}{m_S}\rho_{2n}^{\,n}.
\]
For the mixed term, centering followed by the Cauchy-Schwarz inequality, which separates the numerator from the denominator, gives
\[
\left\|
\frac{UV^n-\mathbf{E}(UV^n)}{m_S^{n+1}}
\right\|_2
\le
\frac{\{\mathbf{E}(U^2V^{2n})\}^{1/2}}{m_S^{n+1}}
\le
\frac{(\mathbf{E}U^4)^{1/4}}{m_S}
\left(\mathbf{E}|V/m_S|^{4n}\right)^{1/4}
=
\frac{\lVert U\rVert_4}{m_S}\rho_{4n}^{\,n},
\]
which is part (i). Under MCAR, $V=\sum_i\beta_i(s_i-\alpha)$ is a sum of independent, centered, bounded terms. Such a sum behaves like a Gaussian variable in the sense that all its moments are controlled by its standard deviation: by a standard moment inequality, or by Hoeffding's lemma, for every fixed $p\ge2$ there is a constant $C_p$, depending only on $p$, with $\lVert V\rVert_p\le C_p(\sum_i\beta_i^2)^{1/2}=C_p\sqrt{W_2}$. Since $\lVert V\rVert_2=\sqrt{\alpha(1-\alpha)W_2}$, this gives $\lVert V\rVert_p\le C_{p,\alpha}\lVert V\rVert_2$ with $C_{p,\alpha}:=C_p/\sqrt{\alpha(1-\alpha)}$, and, because $\lVert V\rVert_p$ does not decrease with $p$, $\rho\le\rho_p\le C_{p,\alpha}\rho$, which is part (ii). Part (iii) follows from parts (i) and (ii), since under MCAR $|m_R|/m_S=|m_r|$.

The geometric argument requires the largest relative departure of the total reported weight to stay below one. Under MCAR this fails for $\alpha\le\tfrac12$ on the pattern in which every site reports, which is where the truncated variances in Table~\ref{tab:convergence} fail to converge.

\section{Why the truncation error is controlled by the concentration of the denominator (Theorem~\ref{thm:rho-error})}
\label{app:proof-thm}

The error of a truncation is the difference between two variances, which Eq.~\eqref{eq:var-diff} splits into a covariance and a variance involving the omitted part $f-P$. We establish upper bounds on each separately on reporting patterns where the total reported weight is close to its mean, i.e. the good event $G$ where the omitted terms form a geometric tail, and on the remaining patterns that are rare and enter only through their probability.

Since $\mathcal T_{0,n_2}[f]=P-\mathbf{E}P$ and $\mu_2^{(0,n_2)}=\mathrm{Var}(P)$, Eq.~\eqref{eq:var-diff} gives
\begin{equation}
\mathrm{Var}(f)-\mu_2^{(0,n_2)}
=
2\,\mathrm{Cov}(f-P,f)-\mathrm{Var}(f-P),
\label{eq:proof-split}
\end{equation}
so that $|\mathrm{Var}(f)-\mu_2^{(0,n_2)}|\le2|\mathrm{Cov}(f-P,f)|+\mathrm{Var}(f-P)$.

\paragraph{The omitted part on the good event.}
On $G$ we have $|V/m_S|\le\tau<1$, so that $f=\{(m_R+U)/m_S\}\sum_{n\ge0}(-V/m_S)^n$, and the omitted part consists of two geometric tails,
\[
f-P
=
\underbrace{
\sum_{n=1}^{\infty}(-1)^n\frac{UV^n}{m_S^{n+1}}
}_{\text{mixed tail}}
+
\underbrace{
m_R\sum_{n=n_2+1}^{\infty}(-1)^n\frac{V^n}{m_S^{n+1}}
}_{\text{denominator tail}},
\]
each at most its first term divided by $1-\tau$:
\[
|\text{mixed tail}|
\le
\frac{|U||V|}{m_S^2(1-\tau)},
\qquad
|\text{denominator tail}|
\le
\frac{|m_R|}{m_S}
\frac{(|V|/m_S)^{n_2+1}}{1-\tau}.
\]
Squaring, using $(x+y)^2\le2x^2+2y^2$ and $\mathbf{1}_G\le1$, and averaging, the good-event remainder $\delta_G^2=\mathbf{E}[(f-P)^2\mathbf{1}_G]$ satisfies
\[
\delta_G^2
\le
\frac{2}{m_S^2(1-\tau)^2}
\left[
\frac{\mathbf{E}[U^2V^2]}{m_S^2}
+
m_R^2\rho_{2(n_2+1)}^{\,2(n_2+1)}
\right],
\]
since $\mathbf{E}|V/m_S|^{2(n_2+1)}=\rho_{2(n_2+1)}^{\,2(n_2+1)}$. This is Eq.~\eqref{eq:thm-delta}.

\paragraph{The covariance term.}
With $g:=f-\mathbf{E}f$, which has mean zero, $\mathrm{Cov}(f-P,f)=\mathbf{E}[(f-P)g]$. On $G$ the Cauchy-Schwarz inequality bounds this average by $\delta_G\lVert g\rVert_2$. On $G^c$ we use only that the field is bounded, so that $|g|\le2\Lambda$, and apply the CauchySchwarz inequality once more
\[
\big|\mathrm{Cov}(f-P,f)\big|
\le
\delta_G\lVert g\rVert_2
+
2\Lambda\,\mathbf{E}\big[|f-P|\mathbf{1}_{G^c}\big]
\le
\delta_G\lVert g\rVert_2
+
2\Lambda\lVert f-P\rVert_2\,\mathbf{P}(G^c)^{1/2}.
\]

\paragraph{The variance term.}
In the same way, splitting the mean square of the omitted part over $G$ and $G^c$, and using $\mathbf{1}_{G^c}^2=\mathbf{1}_{G^c}$ with the Cauchy--Schwarz inequality on $G^c$,
\[
\mathrm{Var}(f-P)
\le
\mathbf{E}(f-P)^2
=
\delta_G^2
+
\mathbf{E}\big[(f-P)^2\mathbf{1}_{G^c}\big]
\le
\delta_G^2
+
\left\{\mathbf{E}(f-P)^4\right\}^{1/2}
\mathbf{P}(G^c)^{1/2}.
\]

Combining the two bounds with Eq.~\eqref{eq:proof-split} gives
\[
\big|\mathrm{Var}(f)-\mu_2^{(0,n_2)}\big|
\le
2\lVert g\rVert_2\,\delta_G+\delta_G^2
+
\left[
4\Lambda\lVert f-P\rVert_2
+
\left\{\mathbf{E}(f-P)^4\right\}^{1/2}
\right]
\mathbf{P}(G^c)^{1/2},
\]
which is Eq.~\eqref{eq:thm-var-to-delta}, the quantity in square brackets being $C_\Lambda$. This constant is finite because $f$ is bounded, and so is $P$, a polynomial in $U$ and $V$ with $|U|\le2\Lambda$ and $|V|\le1$; Appendix~\ref{app:proof-cor-mar} shows that under MCAR at fixed $\alpha$ the bound does not depend on $N$. The probability of the rare event is bounded by Chebyshev's inequality, $\mathbf{P}(|V|>\tau m_S)\le\mathbf{E}V^2/(\tau m_S)^2=\rho^2/\tau^2$, and, under independent reporting, more sharply by Bernstein's inequality, Eq.~\eqref{eq:bernstein}.

The bound thus has the two components seen in Fig.~\ref{fig:rho_master}(c)-(d): what the truncation omits where the normalized denominator is concentrated near zero, and a contribution from patterns outside the good event, whose probability falls rapidly as the network grows.

\section{Why the truncation error vanishes as the network grows (Corollary~\ref{cor:mar-error})}
\label{app:proof-cor-mar}

For uniform weights the denominator concentrates as the network grows, so both components of the bound of Theorem~\ref{thm:rho-error} vanish. We first show this for any spatial dependence, using only that the field is bounded, and then obtain the sharper rate under local dependence, using the decomposition of Appendix~\ref{app:tools}.

\paragraph{Any spatial dependence.}
Since $|U|\le2\Lambda$, $\mathbf{E}[U^2V^2]\le4\Lambda^2\mathbf{E}V^2=\mathcal O(N^{-1})$, and Eq.~\eqref{eq:V-even-fixedalpha} with $k=n_2+1$ makes the denominator term of Eq.~\eqref{eq:thm-delta} of order $N^{-(n_2+1)}$, so that $\delta_G^2=\mathcal O(N^{-1})$ for every fixed $n_2\ge0$. For uniform weights the largest weight is $1/N$, and Bernstein's inequality gives
\[
\mathbf{P}(G^c)
\le
2\exp(-cN),
\qquad
c=
\frac{\tau^2\alpha}
{2(1-\alpha+\tau/3)}
>0,
\]
for fixed $\alpha\in(0,1)$ and $\tau\in(0,1)$. The coefficient $C_\Lambda$ of this probability does not grow with $N$: the partial sum $P=U/\alpha+m_R\sum_{n=0}^{n_2}(-1)^nV^n/\alpha^{n+1}$ satisfies, since $|U|\le2\Lambda$, $|V|\le1$, and $|m_R|\le\Lambda\alpha$,
\[
|P|
\le
\frac{2\Lambda}{\alpha}
+
\Lambda\sum_{n=0}^{n_2}\alpha^{-n}
=:K_P,
\]
which depends on $\Lambda$, $\alpha$, and $n_2$ but not on $N$. With $|f|\le\Lambda$, $|f-P|\le\Lambda+K_P=:K$, so that $\lVert f-P\rVert_2\le K$, $\{\mathbf{E}(f-P)^4\}^{1/2}\le K^2$, and $C_\Lambda\le4\Lambda K+K^2$. Finally $\lVert f-\mathbf{E}f\rVert_2\le\Lambda$, so that Eq.~\eqref{eq:thm-var-to-delta} gives
\[
\big|\mathrm{Var}(f)-\mu_2^{(0,n_2)}\big|
=
\mathcal O(N^{-1/2})
+
\mathcal O(N^{-1})
+
\mathcal O(\exp(-cN/2))
\longrightarrow0,
\]
which is the first statement of Corollary~\ref{cor:mar-error}, without any restriction on spatial dependence.

\paragraph{Local dependence.}
The sharper rate rests on the size of the coupling moment $\mathbf{E}[U^2V^2]$. With $U=\alpha\xi+\eta$, $U^2V^2\le2\alpha^2\xi^2V^2+2\eta^2V^2$, and the two parts are controlled by different properties. The field part involves the spatial dependence: since $\xi$ and $V$ are independent, $\mathbf{E}[\xi^2V^2]=\mathbf{E}\xi^2\,\mathbf{E}V^2=\mathcal O(N^{-2})$ under the local-dependence condition. The reporting part is controlled by independent reporting and boundedness of the field, whatever the spatial dependence. Given the field, $\eta=N^{-1}\sum_iZ_i$ with $Z_i:=x_ir_i$ independent, centered, and bounded by $\Lambda$, and in the fourth moment of their sum only terms in which every index appears at least twice survive:
\[
\mathbf{E}\!\left[
\Big(\sum_i Z_i\Big)^4
\Bigm| r_1,\ldots,r_N
\right]
=
\sum_i\mathbf{E}(Z_i^4\mid r)
+
6\sum_{i<j}
\mathbf{E}(Z_i^2\mid r)\,
\mathbf{E}(Z_j^2\mid r).
\]
The first sum is $\mathcal O(N)$ and the second $\mathcal O(N^2)$, whatever the field realization, so that dividing by $N^4$ and averaging over the field, $\mathbf{E}\eta^4=\mathcal O(N^{-2})$. With $\mathbf{E}V^4=\mathcal O(N^{-2})$ from Eq.~\eqref{eq:V-even-fixedalpha}, the Cauchy-Schwarz inequality gives $\mathbf{E}[\eta^2V^2]\le(\mathbf{E}\eta^4)^{1/2}(\mathbf{E}V^4)^{1/2}=\mathcal O(N^{-2})$. Hence
\begin{equation}
\mathbf{E}[U^2V^2]
=
\mathcal O(N^{-2})
=
\mathcal O(\rho^4).
\label{eq:U2V2-rate-sec2}
\end{equation}
Substituting Eqs.~\eqref{eq:V-even-fixedalpha} and~\eqref{eq:U2V2-rate-sec2} into Eq.~\eqref{eq:thm-delta}, with $m_S=\alpha$ and $|m_R|\le\Lambda\alpha$,
\[
\delta_G^2
=
\mathcal O(N^{-2})
+
\mathcal O\!\left(N^{-(n_2+1)}\right).
\]
The first term is the coupling between numerator and denominator, and the second the omitted denominator terms. For $n_2=0$ the second is of order $N^{-1}$ and dominates, whereas for $n_2\ge1$ it is of order $N^{-2}$ or smaller, so that
\[
\delta_G
=
\begin{cases}
\mathcal O(N^{-1/2})
=\mathcal O(\rho),
& n_2=0,\\[3pt]
\mathcal O(N^{-1})
=\mathcal O(\rho^2),
& n_2\ge1,
\end{cases}
\]
and more generally the pure-denominator contribution to $\delta_G^2$ is $\mathcal O(\rho^{2(n_2+1)})$. Since $C_\Lambda$ is bounded uniformly in $N$, the rare-event contribution is exponentially smaller. These are the rates stated in Corollary~\ref{cor:mar-error}. The exponential fall of the rare-event probability is the one shown in Fig.~\ref{fig:rho_master}(d), well below the Chebyshev bound.

\section{Why the error of the first-order approximation is of order \texorpdfstring{$\rho^4$}{rho\^{}4} (Proposition~\ref{prop:rho4})}
\label{app:proof-rho4}

For $(0,0)$ it is enough to note that both the exact variance and the approximation are of order $1/N$. For $(0,1)$ the general bound gives only order $\rho^3$, and a sharper rate requires the explicit form of what the truncation omits. The omitted part is of second order in the fluctuations, and its covariance with the retained first-order part reduces to averages of products of three fluctuations, each of order $N^{-2}$. We treat $(0,0)$ first and then $(0,1)$, throughout with uniform weights, fixed $\alpha$, a bounded field, and the local-dependence condition.

\paragraph{The simplest approximation, $(0,0)$.}
Here $\mu_2^{(0,0)}=\mathrm{Var}(R)/\alpha^2$, and with $U=\alpha\xi+\eta$ from Appendix~\ref{app:tools}, $\mathbf{E}[\xi\eta]=0$ and $\mathbf{E}\xi^2=\mathcal O(N^{-1})$. Given the field, $\eta$ averages independent terms, so that
\[
\mathbf{E}\eta^2
=
\frac{\alpha(1-\alpha)}{N^2}
\sum_i\mathbf{E}r_i^2
=
\mathcal O(N^{-1}),
\]
and $\mathrm{Var}(R)=\mathbf{E}U^2=\mathcal O(N^{-1})$. Therefore
\begin{equation}
\mu_2^{(0,0)}
=
\mathcal O(N^{-1}).
\label{eq:mu00-rate}
\end{equation}
The exact variance has the same order. With $\zeta:=N^{-1}\sum_ix_i(r_i-m_r)$, and since $N^{-1}\sum_i(r_i-m_r)=N^{-1}\sum_i(r_i-m_i)=\xi$, we have $R-m_rS=\alpha\xi+\zeta$, with $\mathbf{E}(\zeta\mid r_1,\ldots,r_N)=0$, so that $\mathbf{E}[\xi\zeta]=0$. Independence of the reporting indicators gives
\[
\mathbf{E}\zeta^2
=
\frac{\alpha(1-\alpha)}{N^2}
\sum_i\mathbf{E}(r_i-m_r)^2
=
\mathcal O(N^{-1}),
\]
so that
\begin{equation}
\mathbf{E}(R-m_rS)^2
=
\mathcal O(N^{-1}).
\label{eq:R-minus-mrS}
\end{equation}
On $G$, $S\ge(1-\tau)\alpha$, so that $|f-m_r|=|R-m_rS|/S\le|R-m_rS|/\{(1-\tau)\alpha\}$, while on $G^c$, $|f-m_r|\le2\Lambda$ and $\mathbf{P}(G^c)=\mathcal O(\exp(-cN))$. Hence
\[
\mathrm{Var}(f)
\le
\mathbf{E}(f-m_r)^2
\le
\frac{\mathbf{E}(R-m_rS)^2}
{(1-\tau)^2\alpha^2}
+
4\Lambda^2\mathbf{P}(G^c)
=
\mathcal O(N^{-1}),
\]
and, with Eq.~\eqref{eq:mu00-rate}, $|\mathrm{Var}(f)-\mu_2^{(0,0)}|\le\mathrm{Var}(f)+\mu_2^{(0,0)}=\mathcal O(N^{-1})=\mathcal O(\rho^2)$.

\paragraph{The first-order approximation, $(0,1)$.}
The uncentered partial sum is $P_{01}=m_R/m_S+U/m_S-m_RV/m_S^2$, so that $\mu_2^{(0,1)}=\mathrm{Var}(P_{01})$, with centered part $L:=P_{01}-\mathbf{E}P_{01}=U/m_S-m_RV/m_S^2$, the retained first-order fluctuations of numerator and denominator. Writing $f=P_{01}+\mathcal R$, where $\mathcal R$ is the omitted part,
\begin{equation}
\mathrm{Var}(f)-\mu_2^{(0,1)}
=
2\,\mathbf{E}[L\mathcal R]+\mathrm{Var}(\mathcal R),
\label{eq:01-var-decomp}
\end{equation}
since $\mathbf{E}L=0$. On $G$, subtracting $P_{01}$ from $f=(m_R+U)/(m_S+V)$ gives the omitted part exactly,
\begin{equation}
\mathcal R
=
\frac{
-\dfrac{UV}{m_S^2}
+
\dfrac{m_RV^2}{m_S^3}
}{
1+V/m_S
}.
\label{eq:exact-01-remainder}
\end{equation}
Its numerator is of second order in the fluctuations: the first term is the leading coupling between numerator and denominator that $(0,1)$ omits, and the second the leading quadratic denominator term.

\emph{The variance of the omitted part.} Since $|1+V/m_S|\ge1-\tau$ on $G$, $m_S=\alpha$, and $|m_R|\le\Lambda\alpha$, Eq.~\eqref{eq:exact-01-remainder} gives $\mathcal R^2\mathbf{1}_G\le C(U^2V^2+V^4)\mathbf{1}_G$ with $C$ independent of $N$, so that Eqs.~\eqref{eq:V-even-fixedalpha} and~\eqref{eq:U2V2-rate-sec2} give $\mathbf{E}[\mathcal R^2\mathbf{1}_G]=\mathcal O(N^{-2})$. On $G^c$ no expansion is needed: $|f|\le\Lambda$ and $P_{01}$ is bounded uniformly in $N$, so that $\mathcal R$ is bounded and $\mathbf{E}[\mathcal R^2\mathbf{1}_{G^c}]=\mathcal O(\exp(-cN))$. Hence
\begin{equation}
\mathrm{Var}(\mathcal R)
\le
\mathbf{E}\mathcal R^2
=
\mathcal O(N^{-2}).
\label{eq:Rvar-rate}
\end{equation}

\emph{The covariance of the omitted and retained parts.} A bound using only the sizes of $L$ and $\mathcal R$ would give order $N^{-3/2}$, and the sharper rate comes from the explicit form of $\mathcal R$. Writing $A_0:=-UV/m_S^2+m_RV^2/m_S^3$ for the second-order numerator in Eq.~\eqref{eq:exact-01-remainder}, on $G$
\[
\mathcal R
=
A_0
-
\frac{(V/m_S)A_0}{1+V/m_S},
\]
where the second term carries one further factor of $V$. Multiplying the first-order $L$ by the second-order $A_0$ and averaging gives exactly three third-order moments,
\begin{equation}
\mathbf{E}[LA_0]
=
-\frac{\mathbf{E}[U^2V]}{m_S^3}
+
\frac{2m_R\,\mathbf{E}[UV^2]}{m_S^4}
-
\frac{m_R^2\,\mathbf{E}[V^3]}{m_S^5},
\label{eq:LA0-third-moments}
\end{equation}
which describe, respectively, two numerator fluctuations interacting with one denominator fluctuation, one numerator fluctuation with two denominator fluctuations, and the denominator alone. Under MCAR, Supplementary Information Section~S1 gives $\mathbf{E}[UV^2]=\alpha(1-\alpha)(1-2\alpha)M_3$, $\mathbf{E}[U^2V]=\alpha(1-\alpha)[V_3+\alpha C_\sigma+(1-2\alpha)H_3]$, and $\mathbf{E}V^3=\alpha(1-\alpha)(1-2\alpha)W_3$, which for uniform weights involve
\[
M_3=\frac1{N^3}\sum_i m_i,
\qquad
V_3=\frac1{N^3}\sum_i\sigma_i^2,
\qquad
H_3=\frac1{N^3}\sum_i m_i^2,
\qquad
C_\sigma
=
\frac{2}{N^3}\sum_{i\ne j}\sigma_{ij},
\qquad
W_3=\frac1{N^2}.
\]
Boundedness of the field makes $M_3$, $V_3$, and $H_3$ of order $N^{-2}$, and the local-dependence condition gives $|C_\sigma|\le2N^{-3}\sum_{i\ne j}|\sigma_{ij}|=\mathcal O(N^{-2})$. Since $m_S=\alpha$ and $m_R$ is bounded,
\begin{equation}
\mathbf{E}[LA_0]
=
\mathcal O(N^{-2}).
\label{eq:LA0-rate}
\end{equation}
The further factor of $V$ does not change this order. On $G$,
\[
\left|
L\,
\frac{(V/m_S)A_0}{1+V/m_S}
\right|
\le
C\left(
U^2V^2+|U||V|^3+V^4
\right),
\]
where the first and third terms have mean of order $N^{-2}$, and the Cauchy-Schwarz inequality gives $\mathbf{E}[|U||V|^3]\le\{\mathbf{E}(U^2V^2)\}^{1/2}\{\mathbf{E}V^4\}^{1/2}=\mathcal O(N^{-2})$, so that
\begin{equation}
\mathbf{E}\!\left[
\left|
L\,
\frac{(V/m_S)A_0}{1+V/m_S}
\right|
\mathbf{1}_G
\right]
=
\mathcal O(N^{-2}).
\label{eq:01-denom-correction}
\end{equation}
Separating $G$ and $G^c$, so as not to use the representation of $\mathcal R$ outside $G$,
\[
\mathbf{E}[L\mathcal R]
=
\mathbf{E}[LA_0]
-
\mathbf{E}[LA_0\mathbf{1}_{G^c}]
-
\mathbf{E}\!\left[
L\,\frac{(V/m_S)A_0}{1+V/m_S}\mathbf{1}_G
\right]
+
\mathbf{E}[L\mathcal R\mathbf{1}_{G^c}].
\]
At fixed $\alpha$, $L$, $A_0$, and $\mathcal R$ are bounded uniformly in $N$, so that the two terms supported on $G^c$ are $\mathcal O(\exp(-cN))$, and with Eqs.~\eqref{eq:LA0-rate}-\eqref{eq:01-denom-correction}, $\mathbf{E}[L\mathcal R]=\mathcal O(N^{-2})$. With Eq.~\eqref{eq:Rvar-rate} in Eq.~\eqref{eq:01-var-decomp},
\[
\big|\mathrm{Var}(f)-\mu_2^{(0,1)}\big|
=
\mathcal O(N^{-2})
=
\mathcal O(\rho^4).
\]

Thus the first-order approximation has the upper bound on its truncation error of order $\rho^4$, against $\rho^2$ for the truncation error of the simplest approximation. These are upper bounds, and particular fields may have smaller errors. For the homogeneous benchmark, the exact errors follow these two rates, with slopes $4$ and $2$ against $\rho$ in Fig.~\ref{fig:rho_master}(a).

\section*{CRediT authorship contribution statement}

\textbf{Ashwin K. Seshadri:} Conceptualization; Theory and Methods; Writing: first draft and revisions; Numerical simulation. \textbf{Abhishek Pal Majumder:} Theory and Methods; Writing: first draft and revisions.

\section*{Declaration of competing interest}
The authors have no known competing interests. 

\section*{Funding}
This research did not receive any external funding.

\section*{Declaration of generative AI and AI-assisted technologies in the manuscript preparation process}
The authors used Claude (Anthropic) to identify errors and refine proofs of results developed by the authors and presented in the main text and Appendices. Claude was also used for deriving and verifying the closed-form expressions for $\mu_{1,2}$, $\mu_{2,1}$, and $\kappa$ under the model of MCAR, and in completing the associated algebra in Supplementary Information Section~S1. All derivations were independently verified by the authors, who are responsible for any remaining errors.

\section*{Acknowledgements}
Thanks to S. Lakshmivarahan, Sreekar Vadlamani, and Vishal Vasan for helpful discussions.

\end{document}